\documentclass[11pt]{article}
\usepackage{latexsym, amsfonts, amscd, amssymb, verbatim, amsmath,
enumerate}

\newtheorem{theorem}{Theorem}[section]
\newtheorem{proposition}{Proposition}[section]
\newtheorem{definition}{Definition}[section]

\newtheorem{lemma}{Lemma}[section]
\newtheorem{remark}{Remark}[section]
\newtheorem{corollary}{Corollary}[section]

\newtheorem{example}{Example}[section]

\newcommand{\proofbox}{\hspace{\fill}{$\Box$}}
\newenvironment{proof}{\textbf{Proof}.}{\proofbox}

\def\bar{\overline}
\date{}

\numberwithin{equation}{section}

\begin{document}

\title{Locally Lipschitz stability of solutions to a parametric parabolic optimal control problem with mixed pointwise constraints}
	
\author{Huynh Khanh\footnote{Department of Mathematics, FPT University, Hoa Lac Hi-Tech Park, Hanoi, Vietnam; email: khanhh@fe.edu.vn}}

\maketitle
	
\medskip
	
\noindent {\bf Abstract.} {\small  A class of parametric optimal control problems governed by semilinear parabolic equations with mixed pointwise constraints is investigated. The perturbations  appear in the objective functional,  the state equation and in mixed pointwise constraints. By analyzing regularity and establishing stability condition of  Lagrange multipliers we prove that, if the strictly second-order sufficient condition for the unperturbed problem is valid, then the solutions of the problems as well as the associated Lagrange multipliers are locally Lipschitz continuous functions of parameters.}
	
\medskip


\noindent {\bf  Key words.}  Solution stability, Second-order sufficient optimality conditions, Parametric  parabolic control, Mixed pointwise constraints.
	
\noindent {\bf AMS Subject Classifications.} 49K15, 90C29
	
\vspace{0.3cm}

\section{Introduction}

Let $\Omega$ be an open and  bounded domain in $\mathbb{R}^N$ with $N=2, 3$ and its  boundary $\Gamma=\partial \Omega$ is of class $C^2$. Let $D=H^2(\Omega)\cap H_0^1(\Omega)$ and $H=L^2(\Omega)$. We consider the following parametric semilinear parabolic optimal control problem: for each parameter $w \in L^\infty(Q)$, find a control function $u\in L^\infty(Q)$  and the corresponding state $y\in L^\infty(Q)\cap W^{1, 1}_2(0, T; D, H)$ which solve
\begin{align}
    &J(y, u, w)=\int_Q L(x, t,  y(x, t), u(x, t), w (x, t)) dxdt\to\inf \label{P1}\\
    &{\rm s.t.}\notag\\
    &\frac{\partial y}{\partial t} +Ay + f(x, t, y, w) = u + w \quad \text{in}\quad Q:= \Omega  \times (0, T)\label{P2} \\
    & y(x, t)=0 \quad \text{on}\quad \Sigma=\Gamma\times [0, T], \quad y(x, 0)=y_0(x) \quad \text{on}\quad \Omega \label{P3}\\
    & g(x, t, y(x, t), u(x,t), w(x, t)) \leq 0 \quad \text{a.a.}\quad (x, t)\in Q, \label{P4}
\end{align} 
where $T > 0$ is given,  $y_0\in H^1_0(\Omega)\cap L^\infty(\Omega)$;  $A$ is  an  elliptic operator which is defined by setting
$$
Ay =  - \sum_{i,j = 1}^N D_j\left(a_{ij}(x) D_i y \right),
$$  $L:Q \times {\mathbb R} \times {\mathbb R} \times \mathbb R \to {\mathbb R}$,  $f: Q \times \mathbb R \times \mathbb R \to \mathbb R$ and $g: \Omega\times[0, T]\times \mathbb{R} \times \mathbb R \times \mathbb R \to \mathbb{R}$ are of class $C^2$.

Throughout the paper we denote by $(P(w))$  the problem \eqref{P1}-\eqref{P4} corresponding to parameter  $w \in L^\infty(Q)$, and  by $\Phi (w)$ the feasible set of  $(P(w))$,  that is, set of all couples $(y, u)$ in $(L^\infty(Q)\cap W^{1, 1}_2(0, T; D, H)) \times L^\infty(Q)$ such that (\ref{P2})-(\ref{P4}) are satisfied.  For a fixed parameter $\bar w \in L^\infty(Q)$, we call $(P(\bar w))$ the unperturbed problem. 


The stability analysis of optimal solutions to  optimization problems as well as optimal control problems depending on a parameter  have some important applications in parameter estimation problems (see for instance \cite{K.Ito-1992}) and in numerical methods of finding optimal solutions. In fact, the stability of solutions guarantees that approximate solutions converge to the original solutions of perturbed problems (see for instance \cite{Griesse-2008, Kuzmanovic-2013} and \cite{Kien-2022}).

In the literature on optimal control of partial differential equations, several contributions to the stability of solutions  with respect to perturbations have been  published. For this we refer the reader to  \cite{Alt-2010, Bonnans-2000, Griesse-2006, Griesse-2008, Hinze-2012, Kien-2017, Kien-2019, Kien-2021, Son-2023, Tuan-2023, Casas-2022-3, Corella-2023}  and the references given therein.   Among those papers, Casas and Tr\"{o}ltzsch \cite{Casas-2022-3}   recently considered  the optimal control problem with box constraint, the objective function is quadratic and convex, the state is solution of a semilinear parabolic Neumann equation. Based on different types of sufficient optimality conditions for a local solution of the unperturbed problem, they proved Lipschitz or H\"{o}lder stability with respect to perturbations of the initial data.  In \cite{Corella-2023}, Corella et.al. showed that H\"{o}lder or Lipschitz dependence of the optimal solution on perturbations for parabolic optimal control problems in which the equation and the objective functional are affine with respect to the control.    

 In the present paper  we consider the parametric optimal control problem (\ref{P1})-(\ref{P4}) with the objective functional and constraints in general setting. We shall focus on second-order sufficient optimality conditions for the problem $(P(\bar w))$ and locally Lipschitz stability of solutions as well as associated Lagrange multipliers.  
Namely, we prove that  under the assumption of a second-order sufficient optimality conditions, the optimal solutions and the associated Lagrange multipliers are locally Lipschitz continuous functions of the parameter. Also, we show that if the strict second-order optimality condition for $(P(\bar w))$ is satisfied at $\bar z = (\bar y, \bar u)$, then $\bar z$ is a locally strong solution of $(P(\bar w))$. Therefore, the second-order sufficient optimality conditions for $(P(\bar w))$ play an important role in the stability of solutions of parametric optimal control problems and so they should be studied parallel to the solution stability.

In order to prove the obtained results, we shall modify the method from \cite{Alt-1990} for problem (\ref{P1})-(\ref{P4}) which is governed by semilinear parabolic equation with  nonlinear and mixed pointwise constraints.  Note that  W. Alt  \cite{Alt-1990} studied  a class of  parametric optimal control problems governed by ordinary differential equations  with pure control constraints. Beside the modification of the method from \cite{Alt-1990}, the choice of function space for control variable is also important.  It is known that if control variable  $u \in L^p(Q)$ with $1 \le p < \infty$, then  the cost function $J$ as well as mappings $f$ and $g$ are hardly differentiable with respect to $u$ on spaces $L^p(Q)$.  By this reason, we choose the space $L^\infty(Q)$ for control variable. However, we have to pay some prices for  this choice: the Lagrange multipliers are additive measure rather than functions; the second-order sufficient optimality condition is not satisfied with respect to the maximum norm. To overcome these difficulties, we first show that under certain conditions, the linear functional on $L^\infty(Q)$ can be represented by a density in $L^1(Q)$. We then use the two-norm  method to deal with the second-order optimality conditions, which is given in  Section 6.

The paper is organized as follows.  
In section 2, we set up assumptions and state our main result in Theorem \ref{DinhLy}. 
Some results related to the state equation are presented in section 3. First-order optimality conditions as well as regularity of Lagrange multipliers for problems $(P(\bar w))$ and $(P(w))$ are established in section 4. Section 5 gives  a result on the  stability of Lagrange multipliers. The second-order sufficient conditions for $(P(\bar w))$ are provided in section 6. The proof of the main result is given in section 7.

\section{Assumptions and statement of the main result}

Let $X$ be  a Banach space with the dual $X^*$. Given  a vector $v \in X$ and a number $r > 0$, we denote by $B_X(v, r)$ and $\bar B_X(v, r)$ the open ball and the closed ball with center $v$ and radius $r$, respectively. In some cases, if no confusion, we can write $B(v, r)$ and $\bar B(v, r)$.

 Throughout this paper we assume that $\partial \Omega$ is of class $C^2$ and the initial datum
 \begin{align*}
     y_0\in H^1_0(\Omega)\cap L^\infty(\Omega).
 \end{align*}
 As mentioned in the introduction,   $H=L^2(\Omega)$, $V=H_0^1(\Omega)$ and $D= H^2(\Omega)\cap H_0^1(\Omega)$. The norm and the scalar product in $H$ are denoted by $|\cdot|$ and $(\cdot, \cdot)$,  respectively.  Note that $|\cdot|^2=(\cdot, \cdot)$. It is known that the embeddings 
$$
D\hookrightarrow V\hookrightarrow H
$$ are compact and each space is dense in the following one. Let $H^{-1}(\Omega)$ be the dual of $H_0^1(\Omega)$. We define the following function spaces
\begin{align*}
& H^1(Q)=W^{1,1}_2(Q)=\{y\in L^2(Q):  \frac{\partial y}{\partial x_i}, \frac{\partial y}{\partial t}\in L^2(Q)\},\\
&W(0, T)=\{y\in L^2(0, T; H^1_0(\Omega)): \frac{\partial y}{\partial t} \in L^2(0, T; H^{-1}(\Omega))\},\\
& W^{1,1}_2(0, T; V, H)=\{y\in L^2(0, T; V): \frac{\partial y}{\partial t}\in L^2(0, T; H) \},\\
& W^{1,1}_2(0, T; D, H)=\{y\in L^2(0, T; D): \frac{\partial y}{\partial t}\in L^2(0, T; H) \},
\end{align*} 
and space of parameters $W := L^\infty(Q)$, control space $U := L^\infty(Q)$; state space is considered as
\begin{align*}
    Y := \bigg\{ y \in W^{1,1}_2(0, T; D, H) \cap L^\infty(Q): \quad \frac{\partial y}{\partial t} + Ay \in L^p(0, T; H) \bigg\}.
\end{align*}
Here $Y$ is endowed with the graph norm
\begin{align*}
    \|y\|_Y:= \|y\|_{W^{1, 1}_2(0, T; D, H)} + \|y\|_{L^\infty(Q)} +  \bigg\|\frac{\partial y}{\partial t} + Ay\bigg\|_{L^p(0, T; H)}.
\end{align*}
and it's a Banach space with this norm. Besides, we have the following embedding: 
\begin{align}
\label{keyEmbed2} 
Y \hookrightarrow  W^{1,1}_2(0, T; D, H)\hookrightarrow  C(0, T; V)  \quad {\rm and} \quad  W(0, T) \hookrightarrow C(0, T; H),
\end{align} where $C(0, T; X)$ stands for the space of continuous mappings $v: [0, T]\to X$ with $X$ is a Banach space. In the sequel, we assume that 
\begin{align}
K_\infty = \Big\{u\in U: u(x, t)\leq 0\ {\rm a.a.}\ (x, t)\in Q\Big\}.  
\end{align}

Let us make the following assumptions which are related to the state equation.


\noindent $(\textbf{H1})$ Coefficients ${a_{ij}}=a_{ji} \in {L^\infty }\left( \Omega  \right)$ for every $1 \le i,j \le N$, satisfy the uniform ellipticity conditions, i.e.,  there exists a number $\alpha > 0$ such that 
\begin{align}
\alpha {\left| \xi  \right|^2} \le \sum\limits_{i,j = 1}^N {{a_{ij}}\left( x \right){\xi _i}{\xi _j}} \,\,\,\,\,\,\,\,\forall \xi  \in {{\mathbb R}^N}\,\,\,{\rm {for\,\,\,a.a.}}\,\,\,x \in \Omega .
\end{align}

\noindent $(\textbf{H2})$ The mapping $f: Q \times \mathbb R \times \mathbb R \to {\mathbb R}$ is a Carath\'{e}odory  function of class $C^2$ with respect to variable $(y, w)$;  $f(\cdot, \cdot, 0, \cdot) = 0$ and there exists a number $C_f\in\mathbb{R}$ such that $f_y(x, t, y, w)\geq C_f$. Furthermore,  for each $M>0$, there exists $k_{f,  M}>0$ such that 
\begin{align*}
    &|f(x, t, y, w)| +  |f_y(x, t, y, w)| + |f_{yy}(x, t, y, w)| \le k_{f, M} \\
    &{\rm and} \\
    &|f(x, t, y_1, w_1) - f(x, t, y_2, w_2)|  +  |f_y(x, t, y_1, w_1) - f_y(x, t, y_2, w_2)| \\
    &+ |f_w(x, t, y_1, w_1) - f_w(x, t, y_2, w_2)|  +  |f_{yy}(x, t, y_1, w_1) - f_{yy}(x, t, y_2, w_2)| \\
    &+ |f_{yw}(x, t, y_1, w_1) - f_{yw}(x, t, y_2, w_2)|  +  |f_{ww}(x, t, y_1, w_1) - f_{ww}(x, t, y_2, w_2)| \\
    &\le k_{f, M}\Big(|y_1 - y_2| + |w_1 - w_2| \Big),
\end{align*}
for a.a.  $(x, t) \in Q$, for all $y, w, y_i, w_i \in \mathbb R$ satisfying $|y|, |w|, |y_i|, |w_i| \le M$ with $i = 1, 2$.

\medskip 

Let us give some illustrative examples showing that $f$ satisfies $(H2)$.
\begin{example} {\rm 
    Let $\Omega := \{x = (x_1, x_2) \in \mathbb R^2 : x_1^2 + x_2^2 < 1\}$, $Q := \Omega \times (0, 7)$ and   
    $f: Q \times \mathbb R \times \mathbb R \to {\mathbb R}$, $f(x, t, y, w) = y(w^2 + t^4 + |x|^2)$. Then $f$ satisfies assumption $(H2)$. Indeed, function  $(x, t, y, w)  \mapsto f(x, t, y, w)$   is of class $C^\infty$; $f(x, t, 0, w) = 0$,  $f_y(x, t, y, w) = w^2 + t^4 + |x|^2 \ge 0$ for all $(x, t, w) \in Q \times \mathbb R$. Moreover, for all $(x, t) \in Q$ and $|y|, |w|, |y_i|,| w_i| \le M$, $i = 1, 2$, we have
    \begin{align*}
        &|f(x, t, y, w)| +  |f_y(x, t, y, w)| + |f_{yy}(x, t, y, w)| \\
        &\quad =  (|y| +  1)(w^2 + t^4 + |x|^2) \le (M + 1)(M^2 + 7^4 + 1), \\
        &|f(x, t, y_1, w_1) - f(x, t, y_2, w_2)| = |y_1(w_1^2 + t^4 + |x|^2) - y_2(w_2^2 + t^4 + |x|^2)| \\
        &\quad \le (t^4 + |x|^2)|y_1 - y_2| +  |y_1w_1^2  - y_1w_2^2 | + |y_1w_2^2  - y_2w_2^2 | \\
        &\quad \le (t^4 + |x|^2 + M^2)|y_1 - y_2| + M|w_1^2  - w_2^2 |  
        \le (t^4 + |x|^2 + M^2)|y_1 - y_2| + 2M^2|w_1  - w_2|  \\
        &\quad \le {\rm max}\Big\{7^4 + 1 + M^2, 2M^2\Big\}\Big(|y_1 - y_2| + |w_1 - w_2| \Big), \\
        &|f_y(x, t, y_1, w_1) - f_y(x, t, y_2, w_2)| = |(w_1^2 + t^4 + |x|^2) - (w_2^2 + t^4 + |x|^2)| \le 2M |w_1 - w_2|, \\
        &|f_w(x, t, y_1, w_1) - f_w(x, t, y_2, w_2)| = 2|y_1w_1 - y_2w_2| \le 2M\Big(|y_1 - y_2| + |w_1 - w_2| \Big), \\
        &|f_{yy}(x, t, y_1, w_1) - f_{yy}(x, t, y_2, w_2)| = 0, \quad  
        |f_{yw}(x, t, y_1, w_1) - f_{yw}(x, t, y_2, w_2)| = 2 |w_1 - w_2|, \\
        &|f_{ww}(x, t, y_1, w_1) - f_{ww}(x, t, y_2, w_2)| = 2 |y_1 - y_2|.
    \end{align*}
    Also, it is easy to see that some other examples of functions $f$ satisfying the above assumption are: $f(x, t, y, w) = y^3(w^2 + t^4 + |x|^2)$,  $f(x, t, y, w) = (y + y^3)(w^2 + t^2 + |x|^2)$,... 
    }
\end{example}

Let $\phi: \Omega\times[0, T]\times\mathbb{R}\times\mathbb{R}  \times \mathbb R   \to \mathbb{R}$ be a mapping which stands for $L$ and $g$. We impose the following hypotheses. 

\noindent $(\textbf{H3})$ $\phi$ is a Carath\'{e}odory  function and  for a.a.  $(x,t)\in \Omega\times [0, T]$, $\phi(x, t, \cdot, \cdot, \cdot)$ is of class $C^2$ and satisfies the following property:  for each $M>0$, there exists $k_{\phi,M}>0$ such that 
\begin{align*}
    &|\phi_y(x, t, y, u, w)|  + |\phi_u(x, t, y, u, w)| \\ 
    &+ |\phi_{yy}(x, t, y, u, w)| + |\phi_{yu}(x, t, y, u, w)| + |\phi_{uu}(x, t, y, u, w)| \le k_{\phi, M} \\
    &{\rm and} \\
    &|\phi(x, t, y_1, u_1, w_1) - \phi(x, t, y_2, u_2, w_2)| + |\phi_y(x, t, y_1, u_1, w_1) - \phi_y(x, t, y_2, u_2, w_2)|\\
    &\quad \quad + |\phi_u(x, t, y_1, u_1, w_1) - \phi_u(x, t, y_2, u_2, w_2)| + |\phi_w(x, t, y_1, u_1, w_1) - \phi_w(x, t, y_2, u_2, w_2)|\\
    &\quad \quad + |\phi_{yy}(x, t, y_1, u_1, w_1) - \phi_{yy}(x, t, y_2, u_2, w_2)| + |\phi_{yu}(x, t, y_1, u_1, w_1) - \phi_{yu}(x, t, y_2, u_2, w_2)|\\
    &\quad \quad + |\phi_{yw}(x, t, y_1, u_1, w_1) - \phi_{yw}(x, t, y_2, u_2, w_2)| + |\phi_{uu}(x, t, y_1, u_1, w_1) - \phi_{uu}(x, t, y_2, u_2, w_2)|\\
    &\quad \quad + |\phi_{uw}(x, t, y_1, u_1, w_1) - \phi_{uw}(x, t, y_2, u_2, w_2)| + |\phi_{ww}(x, t, y_1, u_1, w_1) - \phi_{ww}(x, t, y_2, u_2, w_2)|\\
    &\quad \quad \le k_{\phi,M}\Big(|y_1-y_2|+ |u_1-u_2| + |w_1- w_2| \Big),
\end{align*}
for a.a.  $(x, t)\in \Omega\times[0, T]$ and  for all $y, u, w, y_i, u_i, w_i \in \mathbb{R}$ satisfying $|y|, |u|, |w|, |y_i|, |u_i|, |w_i| \leq M$ with $i=1,2$. 

\medskip

We now fix $\bar w \in W$. Let $(\bar y, \bar u) \in \Phi(\bar w)$. 

\noindent $(\textbf{H4})$ There exists a constant $\gamma_0 > 0$ such that 
\begin{align}
    \Big|g_u(x, t, \bar y(x, t) , \bar u(x, t), \bar w(x, t))\Big|  \ge \gamma_0 \quad {\rm a.a.} \quad (x, t) \in Q.
\end{align}

 \medskip

   In the above assumptions,  $(H1)$ and $(H2)$ imply that for each parameter $\bar w \in W$ and each control $\bar u \in U$,  there exists a unique state $\bar y \in Y$, see Lemma \ref{Lemma-stateEq} below. Moreover, hypothesis $(H2)$ makes sure that $f$ is of class $C^2$ on $Y \times W$.  Hypothesis $(H3)$ makes sure that $J$ and $g$ are of class $C^2$ on $Y\times U \times W$. Meanwhile, $(H4)$ guarantees that the Robinson constraint qualification is satisfied and the  associated Lagrange multipliers belong to $L^\infty(Q)$.     

 Let us give an illustrative example showing that $g$ satisfies assumptions $(H3)$ and $(H4)$.  

\begin{example}{\rm 
The following formulae of $g$ satisfies assumptions $(H3)$ and  $(H4)$:
\begin{align*}
    & g(x, t, y, u, w) = u + w, \\
    & g(x, t, y, u, w) = y^3  +  u^3  + w^3 + u, \\
    & g(x, t, y, u, w) = \frac{1}{3}y^2u^3 - yu^2 + (2 + |t-1| + w^2)u, \\
    & g(x, t, y, u, w) = \pi(x, t) + w^4u^3 + (y^2 + 1)u
\end{align*}
where $\pi$ is a continuous function on $Q$. 
}
\end{example}

For each $w \in W$, recall that $\Phi(w)$ is the feasible set of the problem $(P(w))$, that is, set of all couples $(y, u) \in Y \times U$ such that (\ref{P2})-(\ref{P4}) are satisfied. Let $\bar z = (\bar y, \bar u) \in \Phi(\bar w)$. Then, symbols $f[x, t]$, $f_y[x, t]$, $\phi[x, t]$, $\phi_y[x, t]$, $\phi_u[\cdot, \cdot]$, etc., stand for  $f(x, t, \bar y(x, t), \bar w(x, t))$, $f_y(x, t, \bar y(x, t), \bar w(x, t))$, $\phi(x, t, \bar y(x, t), \bar u(x, t), \bar w(x, t))$, $\phi_y(x, t, \bar y(x, t), \bar u(x, t), \bar w(x, t))$, \\
$\phi_u(\cdot, \cdot, \bar y(\cdot, \cdot), \bar u(\cdot, \cdot), \bar w(\cdot, \cdot)$, etc.,  respectively.  Also, given a couple  $(\widehat y_w, \widehat u_w) \in \Phi(w)$, symbols  $f[x, t, w]$, $f_y[x, t, w]$, $\phi[x, t, w]$, $\phi_y[x, t, w]$, $\phi_u[\cdot, \cdot, w]$, etc., stand for  $f(x, t, \widehat y_w(x, t),  w(x, t))$, \\
$f_y(x, t, \widehat y_w(x, t),  w(x, t))$, 
$\phi(x, t, \widehat y_w(x, t), \widehat u_w(x, t),  w(x, t))$, $\phi_y(x, t, \widehat y_w(x, t), \widehat u_w(x, t),  w(x, t))$, \\ 
$\phi_u(\cdot, \cdot, \widehat y_w(\cdot, \cdot), \widehat u_w(\cdot, \cdot),  w(\cdot, \cdot)$, etc.,  respectively.

\begin{definition}
\label{df.00}
    The functions $\varphi_w \in W^{1, 1}_2(0, T; D, H) \cap L^\infty(Q)$ and $e_w \in  L^\infty(Q)$ are said to be Lagrange multipliers of the problem $(P(w))$ at $(\widehat y_w, \widehat u_w) \in \Phi(w)$ if they satisfy the following conditions:
    
\noindent $(i)$ (the adjoint equation) 
$$
-\dfrac{\partial \varphi_w}{\partial t} + A^* \varphi_w + f_y[\cdot, \cdot, w]\varphi_w = -L_y[\cdot, \cdot, w] - e_wg_y[\cdot, \cdot, w], \quad \varphi_w(., T) = 0, 
$$ where $A^*$ is the adjoint operator of $A$, which is defined by 
$$
A^*\varphi_w =  - \sum_{i,j = 1}^N D_i\left(a_{ij}\left( x \right) D_j \varphi_w \right);
$$ 
 
 \noindent $(ii)$ (the stationary condition in $\widehat u_w$) 
$$
L_u[x, t, w] - \varphi_w(x,t) + e_w(x,t)g_u[x, t, w] = 0 \ \ {\rm a.a.} \  (x, t) \in Q;
$$

\noindent $(iii)$ (the complementary condition) 
\begin{align}\label{ComlementCond}
e_w(x,t)g[x, t, w]=0  \quad   {\rm and} \quad   e_w(x, t)\geq 0\quad {\rm a.a.} \ (x,t)\in Q.
\end{align}    
\end{definition}
Denote $\Lambda_\infty [(\widehat y_w, \widehat u_w), w]$, $\Lambda_\infty [(\bar y, \bar u), \bar w]$ by  the sets of Lagrange multipliers of  $(P(w))$ at $(\widehat y_w, \widehat u_w)$ and of $(P(\bar w))$ at $(\bar y, \bar u)$, respectively. 

\begin{definition}
    The point $\bar z = (\bar y, \bar u) \in \Phi(\bar w)$, is said to be a locally strongly optimal solution of the problem $(P(\bar w))$ if there exist numbers $\epsilon>0$ and $\kappa > 0$ such that 
\begin{align}
\label{strong.solution}
    J(y, u, \bar w) \ge  J(\bar y, \bar u, \bar w) +\kappa \|u-\bar u\|_{L^2(Q)}^2\quad \forall (y, u)\in \Phi(\bar w) \cap[B_Y(\bar y, \epsilon)\times B_U(\bar u, \epsilon)], 
\end{align} where $B_Y(\bar y, \epsilon)$ and $B_U(\bar u, \epsilon)$ are  balls in $Y$ and $L^\infty(Q)$, respectively.
\end{definition}

We now denote  by $\mathcal{C}_2[\bar z, \bar w]$ the set of all couples $(y, u) \in W^{1, 1}_2(0, T; D, H) \times L^2(Q)$ satisfying the following condition:
\begin{align}
    \dfrac{\partial y}{\partial t} + Ay + f_y[x, t]y = u, \quad y(0) = 0.
\end{align}

The following theorem is the main result of the paper.

\begin{theorem}
\label{DinhLy}
    Let $\bar z = (\bar y, \bar u) \in \Phi(\bar w)$. Suppose that assumptions $(H1)$-$(H4)$ are valid and there exists a couple $(\varphi, e) \in \Lambda_\infty [\bar z, \bar w]$ such that the following strictly second-order condition:
    \begin{align}
     \label{StrictSOSCond.0}
    &\int_Q(L_{yy}[x, t]y^2 + 2L_{yu}[x, t]yu  + L_{uu}[x, t]u^2)dxdt \nonumber \\
    &+ \int_Q e(x, t)\Big( g_{yy}[x, t]y^2 + 2g_{yu}[x, t]yu  + g_{uu}[x, t]u^2 \Big) dxdt  + \int_Q \varphi f_{yy}[x, t] y^2 dxdt > 0,
\end{align}
is satisfied for all $(y, u)\in\mathcal{C}_2[(\bar y, \bar u), \bar w] \setminus \{(0,0)\}$. Furthermore, there exists a number $\varrho > 0$ such that 
\begin{align}
    \label{StrictSOSCond.01}
    L_{uu}[x, t] + e(x, t)g_{uu}[x, t] \ge \varrho \quad {\rm a.e.} \ (x, t) \in Q.
\end{align}
Then $(\bar y, \bar u)$  is a locally strongly optimal solution of the problem $(P(\bar w))$ in the sense of (\ref{strong.solution}).  
Moreover, there exist positive numbers $r_*$  and $s_*$ such that the following assertion is fulfilled:  if $z_w = (\widehat y_w, \widehat u_w)$ is a locally optimal solution of the problem $(P(w))$ with $(\widehat u_w, w) \in B_U(\bar u, r_*) \times B_W(\bar w, s_*)$,  and  $(\varphi_w, e_w) \in \Lambda_\infty [(\widehat y_w, \widehat u_w), w]$  then one has
\begin{align}
    &\|\widehat y_w  - \bar y \|_{W^{1, 1}_2(0, T; D, H)}  + \|\widehat u_w - \bar u\|_{L^2(Q)} \le K_{ lips}. \|w - \bar w\|_{L^\infty(Q)},    \label{KQC} \\
    &\|\varphi_w  -  \varphi\|_{L^2(Q)}   +  \|e_w - e\|_{L^2(Q)}  \le k_{ lips}.\|w - \bar w\|_{L^\infty(Q)},  \label{KQC.1}
\end{align}
for some positive constants $K_{ lips}, k_{lips} > 0$ which are independent of $w$.

\end{theorem}
\begin{remark} {\rm   The Lagrange function associated with the problem (\ref{P1})-(\ref{P4}) is given by 
\begin{align}
    &\mathcal{L}: Y \times U \times L^q(0, T; H) \times \big( L^\infty(\Omega) \cap H_0^1(\Omega) \big)^* \times (L^\infty(Q))^*   \times W \to  \mathbb R, \nonumber \\
    &\mathcal{L}(y, u, \phi_1, \phi_2^*, e^*, w) = \int_QL(x, t, y, u, w)dxdt + \int_Q\bigg[\phi_1 (\frac{\partial y}{\partial t} + Ay + f(x, t, y, w) -u - w) \bigg]dxdt \nonumber \\
    &\quad  \quad + \left\langle \phi_2^*, y(0) - y_0 \right\rangle_{\big( L^\infty(\Omega) \cap H_0^1(\Omega) \big)^*, ( L^\infty(\Omega) \cap H_0^1(\Omega) } + \left\langle e^*, g(\cdot, \cdot, y, u, w)\right\rangle_{L^\infty(Q)^*, L^\infty(Q)},  \label{Lagrange.Function}
\end{align}
where $\frac{1}{p} + \frac{1}{q} = 1$. The fact that  Lagrange multipliers are elements of  $L^q(0, T; H) \times \big( L^\infty(\Omega) \cap H_0^1(\Omega) \big)^* \times (L^\infty(Q))^*$.  In Proposition \ref{4.01} and \ref{4.03} below, we show that the set of all Lagrange multipliers is a subset of $\Big(W^{1, 1}_2(0, T; D, H) \cap L^\infty(Q)\Big) \times L^\infty(Q)$. Then (\ref{StrictSOSCond.0}) becomes $$D^2_{(y,u)}\mathcal{L}(\bar y, \bar u, \varphi, e, \bar w)[(y, v), (y, v)] > 0, \quad \quad \forall (y, v)\in\mathcal{C}_2[(\bar y, \bar u), \bar w]\setminus\{(0,0)\}.$$    The condition (\ref{StrictSOSCond.01}) guarantees that the quadratic form  $D^2_{(y,u)}\mathcal{L}(\bar y, \bar u, \varphi, e, \bar w)[(y, v), (y, v)]$ is sequentially weakly lower semi-continuous in $v$, namely,  the  functional 
\begin{align*}
\zeta \mapsto  \int_Q \big(L_{uu}[x,t] + eg_{uu}[x, t]\big) \zeta^2dxdt
\end{align*} is convex and so it is sequentially lower semi-continous. 
}    
\end{remark}

\section{Analysis of the state equation}

Given  $y_0\in H$, $w \in W$ and $u\in L^2(0, T; H)$,  a function $y\in W(0, T)$ is said to be a {\it weak solution} of the semilinear parabolic equation \eqref{P2}-\eqref{P3} if 
$$
\langle \frac{\partial y}{\partial t}, v\rangle + \sum_{i,j=1}^n\int_\Omega a_{ij}D_iy D_j v dx +(f(x, t, y, w), v) =(u  + w, v), \quad \forall v\in H_0^1(\Omega)
$$ and a.a. $t\in [0, T]$, and $y(0)=y_0$. If a  weak solution $y$ such that $y\in W^{1, 1}_2(0, T; D, H)$ and $f(x, t, y, w)\in L^2(0, T; H)$ then we have $\frac{\partial y}{\partial t} +Ay +f(x, t, y, w)\in L^2(0, T; H)$ and
$$
\langle \frac{\partial y}{\partial t}, v\rangle + (Ay, v) +(f(x, t, y, w), v) =(u  + w, v), \quad \forall v\in H_0^1(\Omega).
$$ Since $H_0^1(\Omega)$ is dense in $H=L^2(\Omega)$, we obtain  
$$
( \frac{\partial y}{\partial t}, v) + (Ay, v) +(f(x, t, y, w), v) =(u  + w, v), \quad \forall v\in H.
$$  Hence
$$
\frac{\partial y}{\partial t} + Ay + f(x, t, y, w) =u  + w \quad \text{a.a.}\ t\in [0, T],\quad  y(0)=y_0.
$$ In this case we say $y$ is a {\it strong solution} of \eqref{P2}-\eqref{P3}. From now on a  solution to \eqref{P2}-\eqref{P3} is understood a strong solution. The existence of strong solutions is provided in the following lemma.


\begin{lemma}
\label{Lemma-stateEq}
Suppose that $(H1)$ and $(H2)$ are satisfied and $y_0\in L^\infty(\Omega)\cap H_0^1(\Omega)$. Then for each $w \in L^\infty(Q)$,  $u\in L^p(0, T; H)$ with $p>\frac{4}{4-N}$, the state equation \eqref{P2}-\eqref{P3} has a unique  solution   $y\in Y$ and there exist positive constants $C_1>0$ and $C_2>0$ such that 
\begin{align}
    & \|y\|_{L^\infty(Q)}\leq C_1 \Big(\|u\|_{L^p(0, T; H)} + \|w\|_{L^\infty(Q)} + \|y_0\|_{L^\infty(\Omega)}\Big), \label{KeyInq0}  \\
    & \Big\| \frac{\partial y}{\partial t}\Big\|_{L^2(0, T; H)} +\|y\|_{L^2(0, T; D)} \le  C_2 \Big(\|u\|_{L^p(0, T; H)} + \|w\|_{L^\infty(Q)} +\|y_0\|_{L^\infty(\Omega)} +\|y_0\|_V\Big).  \label{KeyInq1}
\end{align}  
\end{lemma}
\begin{proof}
    For each $w \in L^\infty(Q)$ and $u\in L^p(0, T; H)$  with $p>\frac{4}{4-N}$, we have $u +  w \in L^p(0, T; H)$.  Applying \cite[Theorem 2.1]{Casas-2020} and \cite[Theorem 2.1]{Casas-2022-1}, the equation \eqref{P2}-\eqref{P3} has a unique solution $y\in L^\infty(Q)\cap W(0, T)$ and inequality \eqref{KeyInq0} is fulfilled. Here, to obtain \eqref{KeyInq0} we use the first estimate in Theorem 2.1 of \cite{Casas-2020} and notice that $f(\cdot, \cdot, 0, \cdot) = 0$.
    
    On the other hand, $|y(x, t)|, |w(x, t)| \le M := \| y\|_{L^\infty(Q)} + \| w\|_{L^\infty(Q)} $ for a.a. $(x, t) \in Q$. Hence, from  assumption $(H2)$  there exists a constant $k_{f, M} > 0$  such that 
    \begin{align}
        |f(x, t, y(x, t), w(x, t))| &= |f(x, t, y(x, t), w(x, t)) - f(x, t, 0, w(x, t))| \nonumber \\ 
        &\le k_{f,M}|y(x, t)| \le k_{f,M}\| y\|_{L^\infty(Q)},
    \end{align}
    for a.a. $(x, t) \in Q$. It follows that $f(\cdot, \cdot, y, w)\in L^\infty(Q)$, and so $\eta_1(\cdot,  \cdot, y, u, w) :=  u + w - f(\cdot, \cdot, y, w)\in L^p(0, T; H)$. These facts allow us to apply improved regularity of solutions to the following parabolic equation 
    \begin{align}
    \label{key.equ1}
        \frac{\partial y}{\partial t} + Ay = \eta_1,\quad y(0) = y_0. 
    \end{align}
According to \cite[Theorem 5, p. 360]{Evan}, the (\ref{key.equ1}) has a unique solution $y\in W^{1, 1}_2(0, T; D, H)\cap L^\infty(0, T; V)$ and the following estimate is satisfied for some constant $C_2' > 0$:
\begin{align}
    \label{KeyInq1.1}
    \Big\| \frac{\partial y}{\partial t}\Big\|_{L^2(0, T; H)} +\|y\|_{L^2(0, T; D)} \le  C_2' \Big(\|\eta_1\|_{L^2(0, T; H)}  + \|y_0\|_V\Big).  
\end{align}
Also, we have 
\begin{align}
    \label{KeyInq1.2}
    &\|\eta_1\|_{L^2(0, T; H)} \nonumber \\ 
    &\le \|u\|_{L^2(0, T; H)} + \|w\|_{L^2(0, T; H)} + \|f(\cdot, \cdot, y, w)\|_{L^2(0, T; H)} \nonumber \\
    &\le C_4 \|u\|_{L^p(0, T; H)} + C_5\|w\|_{L^\infty(Q)} + C_5\|f(\cdot, \cdot, y, w)\|_{L^\infty(Q)} \nonumber \\
    &\le C_4 \|u\|_{L^p(0, T; H)} + C_5\|w\|_{L^\infty(Q)} + C_5 k_{f,M}\|y\|_{L^\infty(Q)} \nonumber \\
    &\le C_4 \|u\|_{L^p(0, T; H)} + C_5\|w\|_{L^\infty(Q)} 
    + C_5 k_{f,M}C_1 \Big(\|u\|_{L^p(0, T; H)} + \|w\|_{L^\infty(Q)} + \|y_0\|_{L^\infty(\Omega)}\Big),
\end{align} 
for some constants $C_4, C_5 > 0$. From (\ref{KeyInq1.1}) and (\ref{KeyInq1.2}) we obtain (\ref{KeyInq1}).   Moreover, we have $\frac{\partial y}{\partial t} + Ay = u + w - f(\cdot, \cdot, y, w) \in L^p(0, T; H)$. This implies that the solution $y \in Y$. 

The lemma is proved.
\end{proof}

\begin{lemma}
\label{Lemma-stateEq.01}
Suppose that $(H1)$ and $(H2)$ are satisfied and $y_0\in L^\infty(\Omega)\cap H_0^1(\Omega)$. Let  $(\widehat u, \widehat w) \in U \times W$ and assume that   $(u, w) \in B_U(\widehat u, r) \times B_W(\widehat w, s)$ for some positive constants $r, s > 0$. Then there exists a constant $C > 0$ such that 
\begin{align}
    \label{KeyInq2}
    \|y - \widehat y\|_Y  \le C \Big(\|u - \widehat u\|_U + \|w - \widehat w\|_W \Big)
\end{align}
where $y, \widehat y$ are solutions to (\ref{P2})-(\ref{P3}) with $(u, w) = (u, w)$, $(u, w) = (\widehat u, \widehat w)$, respectively;  and $\|\cdot\|_Y = \|\cdot\|_{W^{1, 1}_2(0, T; D, H)} + \|\cdot\|_{L^\infty(Q)} + \|\frac{\partial(\cdot)}{\partial t} + A(\cdot)\|_{L^p(0, T; H)}$, $\|\cdot \|_U = \|\cdot\|_W = \|\cdot\|_{L^\infty(Q)}$.
\end{lemma}
\begin{proof}
    Since $y, \widehat y$ are solutions to (\ref{P2})-(\ref{P3}) with $(u, w) = (u, w)$ and  $(u, w) = (\widehat u, \widehat w)$, respectively, 
    \begin{align}
        \label{3.1}
        \begin{cases}
            \dfrac{\partial y}{\partial t}+ A y +f(x, t, y,  w) =  u +   w, \quad  y(0)=y_0,  \\
            \dfrac{\partial \widehat y}{\partial t}+ A \widehat y +f(x, t, \widehat y,  \widehat w) =  \widehat u +   \widehat w, \quad  \widehat y(0)=y_0. 
        \end{cases}
    \end{align}
    This implies that 
    \begin{align}
        \label{3.2}
        \begin{cases}
        \dfrac{\partial ( y- \widehat y_)}{\partial t}+ A(y- \widehat y) + f(x, t,   y,  w) - f(x, t, \widehat y, w)  \\
      \hspace{3.5cm } = \Big[ f(x, t, \widehat y, \widehat w) -  f(x, t, \widehat y,  w) \Big]  + (u -  \widehat u) +  (w - \widehat w),  \\  
      (y- \widehat y)(0) = 0. 
        \end{cases}
    \end{align}
By Taylor's expansion of function $y \mapsto f(x, t, y,  w)$, we have 
 \begin{align}
 \label{3.3}
f(x, t,  y, w) - f(x, t, \widehat y,  w) = f_y\Big(x, t, \widehat y +  \theta( y- \widehat y),  w \Big) ( y- \widehat y),\quad  0\leq \ \theta(x, t)\leq 1. 
\end{align}
Such a function $\theta$ does exist and it is measurable. Indeed, we consider the set-valued map $ H: Q\to 2^{\mathbb{R}}$ by setting
\begin{align}
 H(x, t)=\Big\{ \theta\in [0, 1]  &: f(x, t,  y(x,t),  w)-f(x, t, \widehat y(x,t),  w)  \nonumber  \\
&=f_y\Big(x, t, \widehat y(x,t) + \theta( y(x,t)-\widehat y(x,t)), w\Big)( y(x,t)-\widehat y(x,t))\Big\}.  
\end{align}
 By \cite[Theorem 8.2.9]{Aubin}, $H$ is measurable and has a measurable selection, that is, there exists a measurable mapping $ \theta: Q\to [0,1]$ such that $ \theta(x, t)\in  H(x, t)$ for a.a. $(x, t)\in Q.$ 
  From \eqref{3.2} and (\ref{3.3}), we get
     \begin{align}
        \label{3.4}
        \begin{cases}
        \dfrac{\partial ( y- \widehat y_)}{\partial t}+ A(y- \widehat y) + f_y\Big(x, t, \widehat y +  \theta( y- \widehat y),  w \Big) ( y- \widehat y)  \\
      \hspace{3.5cm } = \Big[ f(x, t, \widehat y, \widehat w) -  f(x, t, \widehat y,  w) \Big]  + (u -  \widehat u) +  (w - \widehat w),  \\  
      (y- \widehat y)(0) = 0. 
        \end{cases}
    \end{align}
    Since   $(u, w) \in B_U(\widehat u, r) \times B_W(\widehat w, s)$, 
    \begin{align*}
        &|u(x, t)| \le |u(x, t) - \widehat u(x, t)| + |\widehat u(x, t)| \le \|u - \widehat u\|_U + \|\widehat u\|_U = \|\widehat u\|_U + r;\\
        &|w(x, t)| \le |w(x, t) - \widehat w(x, t)| + |\widehat w(x, t)| \le \|w - \widehat w\|_W + \|\widehat w\|_W = \|\widehat w\|_W + s;
    \end{align*}
    for a.a. $(x, t)\in Q$. On the  other hand, apply $(i)$ of Lemma \ref{Lemma-stateEq}, there exist $c_1, c_2, c_3 > 0$ such that 
    \begin{align*}
        &|y(x, t)| \le  \|y\|_{L^\infty(Q)} \le c_1 (\|u\|_U + \|w\|_W + \|y_0\|_{L^\infty(\Omega)}) \le c_1 (\|\widehat u\|_U + \|\widehat w\|_W + r +s   + \|y_0\|_{L^\infty(\Omega)}),\\
        &|\widehat y(x, t)| \le  \|\widehat y\|_{L^\infty(Q)} \le c_2 (\|\widehat u\|_U + \|\widehat w\|_W + \|y_0\|_{L^\infty(\Omega)}), \\
        &{\rm and} \\
        &|\widehat y(x, t) + \theta(x, t)(y(x, t) - \widehat y(x, t))| \\
        &\le \|\widehat y + \theta(y - \widehat y)\|_{L^\infty(Q)} \le  2  \|\widehat y\|_{L^\infty(Q)} +  \| y\|_{L^\infty(Q)} \le  c_3 (\|\widehat u\|_U + \|\widehat w\|_W + r +s   + \|y_0\|_{L^\infty(\Omega)})
    \end{align*}
    for a.a. $(x, t)\in Q$. By $(H2)$, there exists $k_{f,M} > 0$ such that 
\begin{align*}
    &\Big|f_y\Big(x, t, \widehat y(x, t) +  \theta(x, t) ( y(x, t)-\widehat y(x, t)), w(x, t)\Big)\Big| \\ 
    &\le  k_{f,M}\Big(|\widehat y(x, t) +  \theta(x, t) ( y(x, t)-\widehat y(x, t))| +  | w(x, t)| \Big) + |f_y(x, t, 0, 0)| \\
    &\le  k_{f,M}\Big(c_3 (\|\widehat u\|_U + \|\widehat w\|_W + r +s   + \|y_0\|_{L^\infty(\Omega)}) +  \|\widehat w\|_W + s   \Big) + \|f_y(\cdot, \cdot, 0, 0)\|_{L^\infty(Q)}
\end{align*}
for a.a. $(x, t)\in Q$. Hence 
\begin{align}
    \label{3.5}
    f_y\Big(x, t, \widehat y +  \theta( y- \widehat y),  w\Big) \in L^\infty(Q).
\end{align}
Also, we have 
\begin{align}
    \label{3.6}
    |f(x, t, \widehat y(x, t), \widehat w(x, t)) -  f(x, t, \widehat y(x, t),  w(x, t)) | \le k_{f,M} |w(x, t) - \widehat w(x, t)| \le k_{f,M} \|w - \widehat w\|_W
\end{align}
for a.a. $(x, t)\in Q$. Hence
\begin{align}
\label{3.7}
    \|f(\cdot, \cdot, \widehat y,  \widehat w) - f(\cdot, \cdot, \widehat y,  w)\|_{L^\infty(Q)} \le k_{f,M} \|w - \widehat w \|_W. 
\end{align}  
This implies that the right-hand side of (\ref{3.4}) belongs to $L^\infty(Q)$. From these facts, applying regularity result from \cite[Theorem 5, p.360]{Evan} to the parabolic equation (\ref{3.4}), there exists a constant $c_4>0$ such that 
\begin{align}
\label{3.8}
        \| y - \widehat y\|_{W_2^{1, 1}(0, T; D, H)} &\le c_4\Big(\|f(\cdot, \cdot, \widehat y,  \widehat w) - f(\cdot, \cdot, \widehat y, w)\|_{L^2(Q)} + \|u - \widehat u\|_{L^2(Q)} +  \|w - \widehat w\|_{L^2(Q)} \Big).
    \end{align}
Combining (\ref{3.7}) with (\ref{3.8}), we obtain
    \begin{align}
    \label{3.9}
        \| y - \widehat y\|_{W_2^{1, 1}(0, T; D, H)} &\le c_4\Big(|Q|^{1/2}k_{f,M} \|w - \widehat w \|_W + \|u - \widehat u\|_{L^2(Q)} +  \|w - \widehat w\|_{L^2(Q)} \Big) \nonumber \\
        &\le c_5 \Big(\|u - \widehat u \|_U +  \|w - \widehat w \|_W \Big)
    \end{align}
for some constants $c_5 > 0$. 

Since the right-hand side of (\ref{3.4}) belongs to $L^\infty(Q) \subset L^p(0, T; H)$, apply $(i)$ of Lemma \ref{Lemma-stateEq} to the equation (\ref{3.4}), there exists a constant $c_6 > 0$ such that 
\begin{align}
    \label{3.10}
    \| y - \widehat y\|_{L^\infty(Q)} &\le c_6\Big(\|f(\cdot, \cdot, \widehat y,  \widehat w) - f(\cdot, \cdot, \widehat y, w)\|_{L^p(0, T; H)} + \|u - \widehat u\|_{L^p(0, T; H)} +  \|w - \widehat w\|_{L^p(0, T; H)} \Big) \nonumber \\
    &\le c_7 \Big(\|u - \widehat u \|_U +  \|w - \widehat w \|_W \Big)
\end{align}
for some constants $c_7 > 0$. Also, from this and the assumption $(H2)$, we have
\begin{align}
    \label{3.11}
    |f(x, t, y(x, t), w(x, t)) &- f(x, t, \widehat y(x, t), w(x, t))| \le k_{f, M}|y(x, t) - \widehat y(x, t)| \nonumber \\
    &\le  k_{f, M} \| y - \widehat y\|_{L^\infty(Q)} \le c_7  k_{f, M} \Big(\|u - \widehat u \|_U +  \|w - \widehat w \|_W \Big)
\end{align}
for a.a. $(x, t)\in Q$. From (\ref{3.2}),  (\ref{3.11}) and (\ref{3.10}). we have 
\begin{align}
    \label{3.12}
    \Big\| \dfrac{\partial ( y- \widehat y_)}{\partial t}+ A(y- \widehat y) \Big\|_{L^p(0, T; H)} &\le \|f(\cdot, \cdot, y, w) - f(\cdot, \cdot, \widehat y, w)\|_{L^p(0, T; H)} \nonumber \\
    &+    \|f(\cdot, \cdot, \widehat y,  \widehat w) - f(\cdot, \cdot, \widehat y, w)\|_{L^p(0, T; H)} \nonumber \\   
    &+ \|u - \widehat u\|_{L^p(0, T; H)} +  \|w - \widehat w\|_{L^p(0, T; H)} \nonumber \\
    &\le (c_7  k_{f, M} |\Omega|^{1/2}T^{1/p} + \frac{c_7}{c_6})\Big(\|u - \widehat u \|_U +  \|w - \widehat w \|_W \Big).
\end{align}
From (\ref{3.9}), (\ref{3.10}) and  (\ref{3.12}), we obtain (\ref{KeyInq2}). This completes the proof of the lemma. 
\end{proof}

\section{First-order optimality conditions and regularity of Lagrange multipliers}

Let us define Banach spaces
\begin{align*}
    &Z = Y \times U,\ E_0 = E_{01} \times E_{02},\quad E = L^\infty(Q), \\
    &E_{01} = L^p(0, T; H), \quad  E_{02} =  L^\infty(\Omega) \cap H_0^1(\Omega) \quad {\rm with} \quad    p > \frac{4}{4 - N}. 
\end{align*}
Define mappings $F: Y\times U \times W \to E_0$ and  $G: Y\times U \times W\to E$ by setting
\begin{align}
     F(y, u, w) = (F_1(y, u, w),\   F_2(y, u, w)) = \bigg(\frac{\partial y}{\partial t} + Ay + f(x, t, y, w) - u - w,\ \   y(0) - y_0 \bigg)
\end{align} and 
\begin{align}   
     G(y, u, w) = g(\cdot, \cdot, y,  u, w).
\end{align}
By definition of space $Y$, if $(y, u, w) \in Y\times U \times W$ then $\frac{\partial y}{\partial t} + Ay \in L^p(0, T; H)$,  $f(x, t, y, w) \in L^\infty(Q) \subset L^p(0, T; H)$ (since $y, w \in L^\infty(Q)$) and $u + w  \in L^\infty(Q) \subset L^p(0, T; H)$. Hence  
$$
\frac{\partial y}{\partial t} + Ay + f(x, t, y, w) - u - w \in L^p(0, T; H) = E_{01}
$$ and $F_1$ is well defined. We now take any  $y \in Y$. By \eqref{keyEmbed2}, we have $y\in C(0, T; V)$. Hence $y(0)\in H_0^1(\Omega)$. Since $y\in L^\infty(Q)$, there exists a constant $M>0$ such that $|y(x, t)|\leq M$ for a.a. $(x, t)\in Q$. Particularly, if $t_n\in [0, T]$ such that $t_n\to 0$, we have 
$|y(x, t_n)|\leq M$ for all $n$ and a.a. $x\in\Omega$ and 
$$
y(\cdot, 0)=\lim_{n\to\infty} y(\cdot, t_n)\quad \text{in}\quad V.
$$ This implies that $\|y(\cdot, 0)-y(\cdot, t_n)\|_{L^2(\Omega)}\to 0$ as $n\to \infty$. By passing to a subsequence, we may assume that $y(x, t_n)\to y(x, 0)$ for a.a. $x\in\Omega$. It follows that 
\begin{align*}
 |y(x, 0)|\leq |y(x, 0)-y(x, t_n)| + |y(x, t_n)|\leq |y(x, 0)-y(x, t_n)| + M.
\end{align*} By passing to the limit when $n\to \infty$, we get $|y(x, 0)|\leq M$. Hence $y(\cdot, 0)\in L^\infty(\Omega)$. Consequently, $y(\cdot, 0)\in H_0^1(\Omega)\cap L^\infty(\Omega)$. Hence  $F_2$ is well defined and so is $F$.

From $(H2)$ and $(H3)$ we see that the mappings $J, F$ and $G$ are of class $C^2$ around $\widehat z_w  = (\widehat y_w, \widehat u_w) \in \Phi(w)$. Their  derivatives $D_{(y, u)}J(\widehat z_w, w)$, $D_{(y, u)}^2J(\widehat z_w, w)$, $D_{(y, u)}F(\widehat z_w, w)$, $D_{(y, u)}^2F(\widehat z_w, w)$,  $D_{(y, u)}G(\widehat z_w, w)$ and $D_{(y, u)}^2G(\widehat z_w, w)$ are given by 

     $D_{(y, u)}J(\widehat z_w, w)(y, u) = \displaystyle \int_Q(L_y[x, t, w]y + L_u[x, t, w]u)dxdt, $
     
    $D_{(y, u)}^2J(\widehat z_w, w)(y, u)^2 = \displaystyle \int_Q(L_{yy}[x, t, w]y^2 + 2L_{yu}[x, t, w]yu + L_{uu}[x, t, w]u^2)dxdt,  $
    
    $D_{(y, u)}F(\widehat z_w, w)(y, u) = (D_{(y, u)}F_1(\widehat z_w, w)(y, u), D_{(y, u)}F_2(\widehat z_w, w)(y, u)) $
    
    ${\rm where} \ D_{(y, u)}F_1(\widehat z_w, w)(y, u) = \dfrac{\partial y}{\partial t} + Ay + f_y[x, t,  w]y - u \ \ {\rm and} \ \ D_{(y, u)}F_2(\widehat z_w, w)(y, u) = y(0), $
    
    $D_{(y, u)}^2F(\widehat z_w, w)(y, u)^2 = \bigg(D_{(y, u)}^2F_1(\widehat z_w, w)(y, u)^2, \ \ D_{(y, u)}^2F_2(\widehat z_w, w)(y, u)^2\bigg) \label{A1.2} $
    
    ${\rm where} \ D_{(y, u)}^2F_1(\widehat z_w, w)(y, u)^2 = f_{yy}[x, t,  w]y^2 \ \ {\rm and} \ \ D_{(y, u)}^2F_2(\widehat z_w, w)(y, u)^2 = 0,$
    
    $D_{(y, u)}G(\widehat z_w, w)(y, u) = g_y[x, t, w]y + g_u[x, t, w]u, $
    
    $D_{(y, u)}^2G(\widehat z_w, w)(y, u)^2 = g_{yy}[x, t, w]y^2 + 2g_{yu}[x, t, w]yu + g_{uu}[x, t, w]u^2,$\\
for all $(y, u) \in Z$. 

We now formulate problem \eqref{P1}-\eqref{P4} in the form
\begin{align}
    &J(y,u, w) \to {\rm min}, \\
    &{\rm s.t.} \nonumber \\
    &F(y,u, w) = 0,\\
    &G(y,u, w) \in K_\infty.
    \end{align}

Based on \cite[Theorem 3.9]{Bonnans-2000}, we shall derive first-order optimality conditions for the problem $(P(\bar w))$ as follows.

\begin{proposition}
\label{4.01}
    Let $\bar z = (\bar y, \bar u) \in \Phi(\bar w)$. Suppose that assumptions $(H1)$-$(H4)$ are valid and $\bar z$ is a locally optimal solution of the problem $(P(\bar w))$.  Then there exist uniquely determined functions  $\varphi \in W^{1, 2}(0, T; D, H) \cap L^\infty(Q)$ and $e\in L^\infty(Q)$ such that the following conditions are satisfied:
        
\noindent $(i)$ (the adjoint equation) 
$$
-\dfrac{\partial \varphi}{\partial t} + A^* \varphi + f_y[\cdot, \cdot]\varphi = -L_y[\cdot, \cdot] - eg_y[\cdot, \cdot], \quad \varphi(., T) = 0, 
$$ where $A^*$ is the adjoint operator of $A$, which is defined by 
$$
A^*\varphi =  - \sum_{i,j = 1}^N D_i\left(a_{ij}\left( x \right) D_j \varphi \right);
$$ 
 
 \noindent $(ii)$ (the stationary condition in $ u$) 
$$
L_u[x, t] - \varphi(x,t) + e(x,t)g_u[x, t] = 0 \ \ {\rm a.a.} \  (x, t) \in Q;
$$

\noindent $(iii)$ (the complementary condition) 
\begin{align}
\label{ComlementCond.00}
e(x,t)g[x, t]=0  \quad   {\rm and} \quad   e(x, t)\geq 0\quad {\rm a.a.} \ (x,t)\in Q.
\end{align} 
\end{proposition}
\begin{proof}
    The reader is referred to \cite[Theorem 3.1, p. 8]{khanh.Kien.1} for the detailed proof  of this result.
\end{proof}

\medskip

For each  $(v, \widehat w) \in U \times W$, we define 
\begin{align}
    \label{SFw}
    S_F^{v, \widehat w} :&= \Big\{\zeta \in Y : \zeta \ \ {\rm solves} \ \ {\rm equation \ (1.2)-(1.3)} \ \ {\rm with} \ \ u = v, \  w = \widehat w \Big\} \nonumber \\
    &= \Big\{\zeta \in Y : F(\zeta, v, \widehat w) = 0 \Big\}.
\end{align}
By Lemma \ref{Lemma-stateEq}, the set $S_F^{v, \widehat w}$  has exactly one element.

\medskip

Next we show that assumption $(H4)$ remains valid under small perturbations.

\begin{lemma}
\label{4.02}
    Suppose that assumptions $(H3)$, $(H4)$ are satisfied and  $\bar z = (\bar y, \bar u)$ is a locally optimal solution of the problem $(P(\bar w))$. There exist positive numbers $r_1 > 0$ and $s_1 > 0$ such that 
    \begin{align}
    \Big|g_u(x, t,  \widehat y_w(x, t) ,  \widehat u_w(x, t),  w(x, t))\Big|  \ge \dfrac{\gamma_0}{2},
\end{align}
for a.a. $(x, t) \in Q$,  for all $(\widehat u_w, w) \in B_U(\bar u, r_1) \times B_W(\bar w, s_1)$ and  $\widehat y_w \in S_F^{\widehat u_w, w}$,  where $\gamma_0$ is given by the assumption $(H4)$.  
\end{lemma}
\begin{proof}
    Take  $r_1, s_1 > 0$ and $(\widehat u_w, w) \in B_U(\bar u, r_1) \times B_W(\bar w, s_1)$, $\widehat y_w \in S_F^{\widehat u_w, w}$. By the assumption $(H4)$,  we have 
    \begin{align}
        \label{4.02.1}
        \gamma_0 &\le  |g_u(x, t,  \bar y(x, t) ,  \bar u_w(x, t),  \bar w(x, t))| \nonumber \\ &\le \Big|g_u(x, t,  \widehat y_w(x, t) ,  \widehat u_w(x, t),  w(x, t))  -  g_u(x, t,  \bar y(x, t) ,  \bar u_w(x, t),  \bar w(x, t))\Big| \nonumber \\ 
        &+ |g_u(x, t,  \widehat y_w(x, t) ,  \widehat u_w(x, t),  w(x, t))|
    \end{align}
    for a.a. $(x, t) \in Q$. Hence, 
    \begin{align}
        \label{4.02.2}
        &\Big|g_u(x, t,  \widehat y_w(x, t) ,  \widehat u_w(x, t),  w(x, t))\Big| \nonumber \\
        &\ge \gamma_0  -  \Big|g_u(x, t,  \widehat y_w(x, t) ,  \widehat u_w(x, t),  w(x, t))  -  g_u(x, t,  \bar y(x, t) ,  \bar u_w(x, t),  \bar w(x, t))\Big| 
    \end{align}
    for a.a. $(x, t) \in Q$.     Since $\widehat u_w \in B_U(\bar u, r_1)$ and  $w \in B_W(\bar w, s_1)$,
    \begin{align*}
        &|\widehat u_w(x, t)| \le \|\widehat u_w\|_{L^\infty(Q)} = \|(\widehat u_w - \bar u) + \bar u\|_{L^\infty(Q)} \le  \|\bar u\|_{L^\infty(Q)} +  \|\widehat u_w - \bar u\|_{L^\infty(Q)}  \le \|\bar u\|_{L^\infty(Q)} + r_1, \\
        &|w(x, t)| \le \|w\|_{L^\infty(Q)} = \|(w - \bar w) + \bar w\|_{L^\infty(Q)} \le \| \bar w\|_{L^\infty(Q)} + \|w - \bar w\|_{L^\infty(Q)}  \le \|\bar w\|_{L^\infty(Q)} + s_1,
    \end{align*}
    for a.a. $(x, t) \in Q$.  $F(\widehat y_w, \widehat u_w,  w) = 0$. Hence,  from  $(i)$ of Lemma \ref{Lemma-stateEq}, there exists a constant $c  > 0$ such that 
    \begin{align*}
        |\widehat y(x, t)| &\le \|\widehat y\|_{L^\infty(Q)} \le c(\|\widehat u_w\|_{L^\infty(Q)} + \| w\|_{L^\infty(Q)} + \|y_0\|_{L^\infty(\Omega)}) \\ 
        &\le c(\|\bar u\|_{L^\infty(Q)} + r_1 + \|\bar w\|_{L^\infty(Q)} + s_1 + \|y_0\|_{L^\infty(\Omega)}), 
    \end{align*}
    for a.a. $(x, t) \in Q$.    By assumption  $(H3)$, there exists a positive constant $k_{g, M} > 0$ such that 
\begin{align}
    &\Big|g_u(x, t,  \widehat y_w(x, t) ,  \widehat u_w(x, t),  w(x, t))  -  g_u(x, t,  \bar y(x, t) ,  \bar u_w(x, t),  \bar w(x, t))\Big|  \nonumber \\
    &\le k_{g, M}\Big(|\widehat y_w(x, t) - \bar y(x, t) | +  |\widehat u_w(x, t) - \bar u(x, t) | + |w(x, t) - \bar w(x, t) |  \Big)  \nonumber  \\
    &\le k_{g, M}\Big(\|\widehat y_w - \bar y\|_{L^\infty(Q)} +  \|\widehat u_w - \bar u\|_{L^\infty(Q)} + \|w - \bar w\|_{L^\infty(Q)} \Big), \label{4.02.3}
\end{align}
for a.a. $(x, t) \in Q$.   On the other hand, it follows from Lemma \ref{Lemma-stateEq.01} that there exists $C > 0$ such that 
\begin{align}
    \label{4.02.4}
    \|\widehat y_w - \bar y\|_{L^\infty(Q)} \le C \Big(  \|\widehat u_w - \bar u\|_{L^\infty(Q)} + \|w - \bar w\|_{L^\infty(Q)}\Big). 
\end{align}
From (\ref{4.02.3}) and (\ref{4.02.4}), we obtain 
\begin{align}
    \label{4.02.5}
    &\Big|g_u(x, t,  \widehat y_w(x, t) ,  \widehat u_w(x, t),  w(x, t))  -  g_u(x, t,  \bar y(x, t) ,  \bar u_w(x, t),  \bar w(x, t))\Big|  \nonumber \\
    &\le k_{g, M}(C +  1)\Big(  \|\widehat u_w - \bar u\|_{L^\infty(Q)} + \|w - \bar w\|_{L^\infty(Q)} \Big) 
    \le k_{g, M}(C +  1)(r_1  + s_1)
\end{align}
for a.a. $(x, t) \in Q$.  We now choose $r_1, s_1 > 0$ small enough such that 
\begin{align}
    \label{4.02.6}
    k_{g, M}(C +  1)(r_1  + s_1) \le \frac{\gamma_0}{2}. 
\end{align}
With $r_1$ and $s_1$ defined as in (\ref{4.02.6}), combining (\ref{4.02.2}), (\ref{4.02.5}) and (\ref{4.02.6}), we get 
\begin{align}
    \Big|g_u(x, t,  \widehat y_w(x, t) ,  \widehat u_w(x, t),  w(x, t))\Big| \ge \gamma_0 - \frac{\gamma_0}{2} = \frac{\gamma_0}{2}
\end{align}
for a.a. $(x, t) \in Q$, which proves the lemma.
\end{proof}

\begin{proposition}
\label{4.03}
    Let $\bar z = (\bar y, \bar u) \in \Phi(\bar w)$. Suppose that assumptions $(H1)$-$(H4)$ are valid and $\bar z$ is a locally optimal solution of the problem $(P(\bar w))$.  There exist positive numbers $r_1 > 0$ and $s_1 > 0$ such that  the following assertion is fulfilled:  if $z_w = (\widehat y_w, \widehat u_w)$ is a locally optimal solution of the problem $(P(w))$ with $(\widehat u_w, w) \in B_U(\bar u, r_1) \times B_W(\bar w, s_1)$,  then  there exist uniquely determined functions  $\varphi_w \in W^{1, 2}(0, T; D, H) \cap L^\infty(Q)$ and $e_w\in L^\infty(Q)$ such that the following conditions are satisfied:
        
\noindent $(i)$ (the adjoint equation) 
$$
-\dfrac{\partial \varphi_w}{\partial t} + A^* \varphi_w + f_y[\cdot, \cdot, w]\varphi_w = -L_y[\cdot, \cdot, w] - e_wg_y[\cdot, \cdot, w], \quad \varphi_w(., T) = 0, 
$$ where $A^*$ is the adjoint operator of $A$, which is defined by 
$$
A^*\varphi_w =  - \sum_{i,j = 1}^N D_i\left(a_{ij}\left( x \right) D_j \varphi_w \right);
$$ 
 
 \noindent $(ii)$ (the stationary condition in $\widehat u_w$) 
$$
L_u[x, t, w] - \varphi_w(x,t) + e_w(x,t)g_u[x, t, w] = 0 \ \ {\rm a.a.} \  (x, t) \in Q;
$$

\noindent $(iii)$ (the complementary condition) 
\begin{align}
\label{ComlementCond.01}
e_w(x,t)g[x, t, w]=0  \quad   {\rm and} \quad   e_w(x, t)\geq 0\quad {\rm a.a.} \ (x,t)\in Q.
\end{align} 
\end{proposition}
\begin{proof}
By Lemma \ref{4.02},  there exist positive numbers $r_1 > 0$ and $s_1 > 0$ such that the assumption $(H4)$ is still satisfied at $(\widehat y_w, \widehat u_w) \in \Phi(w)$. The rest of the proof of Proposition \ref{4.03} can be done in the same way as the proof of Proposition \ref{4.01}.
\end{proof}

\section{Stability of  Lagrange multipliers}

In this section we establish stability estimates for the associated Lagrange multipliers. The result is the following.  
\begin{lemma}
\label{5.0}
    Let the assumptions of Proposition \ref{4.03} be satisfied. Then there exist  constants $M_1, M_2 > 0$ independent of $w$ such that 
    \begin{align}
        &\|e - e_w\|_{L^2(Q)} \le  M_1 \Big( \|\widehat u_w - \bar u\|_{L^2(Q)} + \|w - \bar w\|_{L^\infty(Q)} \Big), \label{5.00.1} \\        
        &\|\varphi_w - \varphi\|_{L^\infty(Q)}  + \|e_w - e\|_{L^\infty(Q)}   \le  M_2 \Big( \|\widehat u_w - \bar u\|_{L^\infty(Q)} + \|w - \bar w\|_{L^\infty(Q)} \Big). \label{5.00.2}
    \end{align} for all $w\in B_W(\bar w, s_1)$. 
\end{lemma}
\begin{proof}  
Let $\bar z = (\bar y, \bar u)$  be locally optimal solution of the problem $(P(\bar w))$, and  $z_w = (\widehat y_w, \widehat u_w)$ is locally optimal solution of the problem $(P(w))$ with $(\widehat u_w, w) \in B_U(\bar u, r_1) \times B_W(\bar w, s_1)$, as in Proposition \ref{4.03}. Then, by Proposition \ref{4.01} and \ref{4.03} there exist $(\varphi, e) \in \Lambda_\infty [(\bar y, \bar u), \bar w]$ and  $(\varphi_w, e_w) \in \Lambda_\infty [(\widehat y_w, \widehat u_w), w]$. This implies that the following conditions are satisfied:
    \begin{align}
        & - \frac{\partial \varphi}{\partial t} + A^* \varphi + f_y[\cdot, \cdot] \varphi = - L_y[\cdot, \cdot] - eg_y[\cdot, \cdot], \quad \quad \varphi (\cdot, T) = 0;   \label{cond.1}  \\
        & - \frac{\partial \varphi_w}{\partial t} + A^* \varphi_w + f_y[\cdot, \cdot, w] \varphi_w = - L_y[\cdot, \cdot, w] - e_wg_y[\cdot, \cdot, w], \quad \quad \varphi_w (\cdot, T) = 0; \label{cond.2} \\
        & L_u[x, t] - \varphi(x,t) + e(x,t)g_u[x, t] = 0 \ \ {\rm a.a.} \  (x, t) \in Q; \label{cond.3}  \\
        & L_u[x, t, w] - \varphi_w(x,t) + e_w(x,t)g_u[x, t, w] = 0 \ \ {\rm a.a.} \  (x, t) \in Q;  \label{cond.4} \\
        & e(x,t)g[x, t]=0 \  \ {\rm and} \  \ e(x, t)\geq 0\quad {\rm a.a.} \ (x,t)\in Q; \\
        & e_w(x,t)g[x, t, w]=0 \  \ {\rm and} \  \ e_w(x, t)\geq 0\quad {\rm a.a.} \ (x,t)\in Q;
    \end{align}
    Subtracting (\ref{cond.1}) from (\ref{cond.2}), we have
    \begin{align}
        &- \frac{\partial (\varphi_w - \varphi)}{\partial t} + A^* (\varphi_w - \varphi) + \Big(f_y[\cdot, \cdot, w] \varphi_w - f_y[\cdot, \cdot] \varphi\Big) \nonumber \\ 
        & \hspace{4.5cm} = - \Big(L_y[\cdot, \cdot, w] - L_y[\cdot, \cdot] \Big) - \Big(e_wg_y[\cdot, \cdot, w] - eg_y[\cdot, \cdot] \Big), \label{Eq.1}  \\
        &\quad \varphi_w - \varphi  =  0 \quad {\rm on} \ \Sigma \quad {\rm and} \quad (\varphi_w - \varphi)(\cdot, T) = 0 \quad {\rm in} \ \Omega.
    \end{align}
     On the other hand, we have   from conditions (\ref{cond.4}) and (\ref{cond.3}) that 
    \begin{align}
    \label{Eq.2}
          e_w = \frac{1}{g_u[\cdot, \cdot, w]}\Big(\varphi_w - L_u[\cdot, \cdot, w] \Big) \quad \quad {\rm and} \quad \quad   e = \frac{1}{g_u[., .]}\Big(\varphi - L_u[., .] \Big).
    \end{align}
    Inserting (\ref{Eq.2}) into (\ref{Eq.1}), we obtain 
    \begin{align*}
        &- \frac{\partial (\varphi_w - \varphi)}{\partial t} + A^* (\varphi_w - \varphi) + \Big(f_y[\cdot, \cdot, w] \varphi_w - f_y[\cdot, \cdot] \varphi\Big) \\
        &\hspace{3cm} = - \Big(L_y[\cdot, \cdot, w] - L_y[\cdot, \cdot] \Big) - \frac{g_y[\cdot, \cdot, w]}{g_u[\cdot, \cdot, w]}\Big(\varphi_w - L_u[\cdot, \cdot, w] \Big) + \frac{g_y[\cdot, \cdot]}{g_u[., .]}\Big(\varphi - L_u[., .] \Big)
    \end{align*}
    or equivalently, 
    \begin{align}
        \label{Eq.3}
        - \frac{\partial (\varphi_w - \varphi)}{\partial t} + A^* (\varphi_w - \varphi) &+ \bigg( f_y[\cdot, \cdot, w] + \frac{g_y[\cdot, \cdot, w]}{g_u[\cdot, \cdot, w]} \bigg) (\varphi_w - \varphi ) \nonumber \\
        & \quad \quad \quad   = - \Big(L_y[\cdot, \cdot, w] - L_y[\cdot, \cdot] \Big) \nonumber \\
        & \quad \quad \quad  +  \frac{g_y[\cdot, \cdot, w]}{g_u[\cdot, \cdot, w]} L_u[\cdot, \cdot, w] -  \frac{g_y[\cdot, \cdot]}{g_u[\cdot, \cdot]} L_u[\cdot, \cdot]   \nonumber \\
        & \quad \quad \quad   - \bigg[\Big(f_y[\cdot, \cdot, w] - f_y[\cdot, \cdot] \Big) + \Big( \frac{g_y[\cdot, \cdot, w]}{g_u[\cdot, \cdot, w]}  -  \frac{g_y[\cdot, \cdot]}{g_u[\cdot, \cdot]} \Big) \bigg] \varphi \nonumber  \\
        & \quad \quad \quad =:  A_1[\cdot, \cdot, w] +  A_2[\cdot, \cdot, w] + A_3[\cdot, \cdot, w], 
    \end{align}
    $\varphi_w - \varphi  =  0$ on $\Sigma$ \  and \  $(\varphi_w - \varphi)(\cdot, T) = 0$ in $\Omega$. 

    Since $\widehat u_w \in B_U(\bar u, r_1)$ and  $w \in B_W(\bar w, s_1)$,
    \begin{align*}
        &|\widehat u_w(x, t)| \le \|\widehat u_w\|_{L^\infty(Q)} = \|(\widehat u_w - \bar u) + \bar u\|_{L^\infty(Q)} \le  \|\bar u\|_{L^\infty(Q)} +  \|\widehat u_w - \bar u\|_{L^\infty(Q)}  \le \|\bar u\|_{L^\infty(Q)} + r_1, \\
        &|w(x, t)| \le \|w\|_{L^\infty(Q)} = \|(w - \bar w) + \bar w\|_{L^\infty(Q)} \le \| \bar w\|_{L^\infty(Q)} + \|w - \bar w\|_{L^\infty(Q)}  \le \|\bar w\|_{L^\infty(Q)} + s_1,
    \end{align*}
    for a.a. $(x, t) \in Q$.  $F(\widehat y_w, \widehat u_w,  w) = 0$. Hence,  from  $(i)$ of Lemma \ref{Lemma-stateEq}, there exists a constant $b_1 > 0$ such that 
    \begin{align*}
        |\widehat y(x, t)| &\le \|\widehat y\|_{L^\infty(Q)} \le b_1(\|\widehat u_w\|_{L^\infty(Q)} + \| w\|_{L^\infty(Q)} + \|y_0\|_{L^\infty(\Omega)}) \\ 
        &\le  b_1(\|\bar u\|_{L^\infty(Q)} + r_1 + \|\bar w\|_{L^\infty(Q)} + s_1 + \|y_0\|_{L^\infty(\Omega)}), 
    \end{align*}
    for a.a. $(x, t) \in Q$.    By assumption  $(H3)$, there exists a positive constant $k_{L, M} > 0$ such that 
\begin{align}
    \Big|A_1[x, t, w]\Big| &=  |L_y(x, t, \widehat y_w, \widehat u_w, w) - L_y(x, t, \bar y, \bar u, \bar w)| \nonumber \\
    &\le k_{L, M}\Big(|\widehat y_w(x, t) - \bar y(x, t) | +  |\widehat u_w(x, t) - \bar u(x, t) | + |w(x, t) - \bar w(x, t) |  \Big)  \label{5.1.1}  \\
    &\le k_{L, M}\Big(\|\widehat y_w - \bar y\|_{L^\infty(Q)} +  \|\widehat u_w - \bar u\|_{L^\infty(Q)} + \|w - \bar w\|_{L^\infty(Q)} \Big), \label{5.1.2}
\end{align}
for a.a. $(x, t) \in Q$. This yields
\begin{align}
    \label{5.2}
    \|A_1[\cdot, \cdot, w]\|_{L^\infty(Q)} \le k_{L, M}\Big(\|\widehat y_w - \bar y\|_{L^\infty(Q)} +  \|\widehat u_w - \bar u\|_{L^\infty(Q)} + \|w - \bar w\|_{L^\infty(Q)} \Big).
\end{align}
Also, 
\begin{align}
    &\Big|A_2[x, t, w]\Big| =  \bigg| \frac{g_y[x, t, w]}{g_u[x, t, w]} L_u[x, t, w] -  \frac{g_y[x, t]}{g_u[x, t]} L_u[x, t] \bigg| \nonumber \\
    &\le \bigg| \frac{g_y[x, t, w]}{g_u[x, t, w]} \bigg| \Big| L_u[x, t, w] - L_u[x, t] \Big| +  \Big|L_u[x, t] \Big| \bigg| \frac{g_y[x, t, w]}{g_u[x, t, w]} - \frac{g_y[x, t]}{g_u[x, t]} \bigg|  \nonumber \\
    &\le \frac{2k_{L, M}k_{g, M}}{\gamma_0}\Big(|\widehat y_w(x, t) - \bar y(x, t) | +  |\widehat u_w(x, t) - \bar u(x, t) | + |w(x, t) - \bar w(x, t) |  \Big)        \nonumber \\
    &+\frac{2k_{L, M}}{\gamma^2_0} \Big(|g_y[x, t, w]g_u[x, t] - g_y[x, t, w]g_u[x, t, w]| + |g_y[x, t, w]g_u[x, t, w] - g_u[x, t, w]g_y[x, t]|  \Big)     \nonumber \\
    &\le \frac{2k_{L, M}k_{g, M}}{\gamma_0}\Big(|\widehat y_w(x, t) - \bar y(x, t) | +  |\widehat u_w(x, t) - \bar u(x, t) | + |w(x, t) - \bar w(x, t) |  \Big)        \nonumber \\
    &+\frac{4k_{L, M}k^2_{g, M}}{\gamma^2_0}\Big(|\widehat y_w(x, t) - \bar y(x, t) | +  |\widehat u_w(x, t) - \bar u(x, t) | + |w(x, t) - \bar w(x, t) |  \Big)      \nonumber \\
    &= \big(\frac{2k_{L, M}k_{g, M}}{\gamma_0} +  \frac{4k_{L, M}k^2_{g, M}}{\gamma^2_0} \big) \big(|\widehat y_w(x, t) - \bar y(x, t) | +  |\widehat u_w(x, t) - \bar u(x, t) | + |w(x, t) - \bar w(x, t) |  \big)  \label{5.3.1}   \\
    &\le \big(\frac{2k_{L, M}k_{g, M}}{\gamma_0} +  \frac{4k_{L, M}k^2_{g, M}}{\gamma^2_0} \big)\Big(\|\widehat y_w - \bar y\|_{L^\infty(Q)} +  \|\widehat u_w - \bar u\|_{L^\infty(Q)} + \|w - \bar w\|_{L^\infty(Q)} \Big) \nonumber \\
    &=: b_2 \Big(\|\widehat y_w - \bar y\|_{L^\infty(Q)} +  \|\widehat u_w - \bar u\|_{L^\infty(Q)} + \|w - \bar w\|_{L^\infty(Q)} \Big) \label{5.3.2}
\end{align}
for a.a. $(x, t) \in Q$. Hence 
\begin{align}
    \label{5.3}
    \|A_2[\cdot, \cdot, w]\|_{L^\infty(Q)} \le b_2\Big(\|\widehat y_w - \bar y\|_{L^\infty(Q)} +  \|\widehat u_w - \bar u\|_{L^\infty(Q)} + \|w - \bar w\|_{L^\infty(Q)} \Big). 
\end{align}
By similar arguments as above, we evaluate
\begin{align}
    \label{5.4}
    \Big|A_3[x, t, w]\Big| &\le b_3 \Big(|\widehat y_w(x, t) - \bar y(x, t) | +  |\widehat u_w(x, t) - \bar u(x, t) | + |w(x, t) - \bar w(x, t) |  \Big)  \\
    &\le b_3\Big(\|\widehat y_w - \bar y\|_{L^\infty(Q)} +  \|\widehat u_w - \bar u\|_{L^\infty(Q)} + \|w - \bar w\|_{L^\infty(Q)} \Big), \nonumber 
\end{align}
where
\begin{align*}
    b_3 :=   \|\varphi\|_{L^\infty(Q)} {\rm max}(k_{f, M}, \  \frac{4k^2_{g, M}}{\gamma^2_0}).
\end{align*}
Put $A[\cdot, \cdot, w] := A_1[\cdot, \cdot, w] + A_2[\cdot, \cdot, w] + A_3[\cdot, \cdot, w]$ and $b_4 := k_{L, M} + b_2 +  b_3$. Then  $A[\cdot, \cdot, w] \in L^\infty(Q)$, which is the right-hand side of the equation (\ref{Eq.3}). Besides, we get
\begin{align}
    \Big|A[x, t, w]\Big| &\le  b_4 \Big(|\widehat y_w(x, t) - \bar y(x, t) | +  |\widehat u_w(x, t) - \bar u(x, t) | + |w(x, t) - \bar w(x, t) |  \Big) \label{5.5.1}  \\
    &\le  b_4\Big(\|\widehat y_w - \bar y\|_{L^\infty(Q)} +  \|\widehat u_w - \bar u\|_{L^\infty(Q)} + \|w - \bar w\|_{L^\infty(Q)} \Big). \label{5.5.2}
\end{align}
Moreover, we have 
\begin{align}
    \Big|f_y[x, t, w] + \frac{g_y[x, t, w]}{g_u[x, t, w]}  \Big| \le k_{f, M} + \frac{2k_{g, M}}{\gamma_0}  \quad {\rm a.a.} \ (x, t) \in Q.
\end{align}
The equation (\ref{Eq.3}) is very  similar to the parabolic equation (3.10) in \cite{Casas-2023-2} and the equation (33) in \cite{Casas-2015}. For (\ref{Eq.3}), we also obtain estimates that are similar to \cite[estimate (3.13) and  estimate  (3.15)]{Casas-2023-2}. Namely, there exist positive constants $b_5, b_6 > 0$ such that 
\begin{align}
    &\|\varphi_w - \varphi\|_{L^\infty(Q)} \le b_5 \|A[\cdot, \cdot, w]\|_{L^\infty(Q)}, \label{5.6.1} \\
    &\|\varphi_w - \varphi\|_{L^2(Q)} \le  b_6 \|A[\cdot, \cdot, w]\|_{L^2(Q)}. \label{5.6.2}
\end{align}
By Lemma \ref{Lemma-stateEq.01}, there exists  $b_7 > 0$ such that 
\begin{align}
    \label{5.7}
    \|\widehat y_w - \bar y\|_{L^\infty(Q)} \le  \|\widehat y_w - \bar y\|_Y \le b_7\Big(  \|\widehat u_w - \bar u\|_{L^\infty(Q)} + \|w - \bar w\|_{L^\infty(Q)} \Big).
\end{align}
From (\ref{5.6.1}), (\ref{5.5.2}) and (\ref{5.7}), we get 
\begin{align}
    \label{5.8}
    \|\varphi_w - \varphi\|_{L^\infty(Q)} \le b_8\Big(  \|\widehat u_w - \bar u\|_{L^\infty(Q)} + \|w - \bar w\|_{L^\infty(Q)} \Big)
\end{align}
for some constant $b_8 > 0$.  By similar arguments as proof of   (\ref{3.7}) and (\ref{3.8}) in   the proof of Lemma \ref{Lemma-stateEq.01}, we deduce that there exists $b_9 > 0$ such that
\begin{align}
    \label{5.9}
    \|\widehat y_w - \bar y\|_{L^2(Q)} \le  \|\widehat y_w - \bar y\|_{W_2^{1, 1}(0, T; D, H)} \le b_9\Big(  \|\widehat u_w - \bar u\|_{L^2(Q)} + \|w - \bar w\|_{L^\infty(Q)} \Big).
\end{align}
From (\ref{5.6.2}), (\ref{5.5.1}) and (\ref{5.9}), we get 
\begin{align}
    \label{5.10}
    \|\varphi_w - \varphi\|_{L^2(Q)} \le b_{10}\Big(  \|\widehat u_w - \bar u\|_{L^2(Q)} + \|w - \bar w\|_{L^\infty(Q)} \Big)
\end{align} for some constant $b_{10} > 0$.

Subtracting (\ref{cond.3}) from (\ref{cond.4}), we obtain 
\begin{align}
    e_w g_u[\cdot, \cdot, w] - e g_u[\cdot, \cdot] = (\varphi_w - \varphi) - (L_u[\cdot, \cdot, w] - L_u[\cdot, \cdot]),
\end{align}
which is equivalent to
\begin{align}
    \label{e.1}
    e_w - e =  \frac{1}{g_u[\cdot, \cdot, w]}\bigg[(\varphi_w - \varphi) - (L_u[\cdot, \cdot, w] - L_u[\cdot, \cdot]) - e\Big( g_u[\cdot, \cdot, w] - g_u[\cdot, \cdot] \Big) \bigg]. 
\end{align}
Hence
\begin{align}
    \label{e.2}
    &|e_w(x, t) - e(x, t)| \nonumber \\ 
    &\le  \frac{2}{\gamma_0} \bigg(|\varphi_w(x, t) - \varphi(x, t)| + |L_u[x, t, w] - L_u[x, t]| + |e(x, t)|\Big| g_u[x, t, w] - g_u[x, t] \Big| \bigg) \nonumber \\
    &\le  \frac{2}{\gamma_0} \bigg(|\varphi_w(x, t) - \varphi(x, t)| \nonumber \\ 
    &\quad + (k_{L, M} + \|e\|_{L^\infty(Q)}k_{g, M})\Big[|\widehat y_w(x, t) - \bar y(x, t) | +  |\widehat u_w(x, t) - \bar u(x, t) | + |w(x, t) - \bar w(x, t) | \Big] \bigg) 
\end{align}
for a.a. $(x, t) \in Q$. From (\ref{e.2}), (\ref{5.8}) and (\ref{5.7}), there exists $b_{11} > 0$ such that 
\begin{align}
    \label{5.11}
    \|e_w - e\|_{L^\infty(Q)} \le b_{11}\Big(  \|\widehat u_w - \bar u\|_{L^\infty(Q)} + \|w - \bar w\|_{L^\infty(Q)} \Big).
\end{align}
By (\ref{5.11}) and (\ref{5.8}), we obtain (\ref{5.00.2}). By (\ref{e.2}), (\ref{5.10}) and (\ref{5.9}), we obtain (\ref{5.00.1}). The proof of the lemma is complete.
\end{proof}

\section{Second-order sufficient conditions for the unperturbed problem}


In this section, we  give second-order sufficient optimality conditions for locally optimal solution to $(P(\bar w))$.  Recall that 
\begin{align*}
    S_F^{v, \widehat w} = \Big\{\zeta \in Y : F(\zeta, v, \widehat w) = 0 \Big\}
\end{align*}
and the Lagrange function now becomes:
\begin{align}
\label{cmm0}
 \mathcal{L}(y, u,  \varphi, e, w)   = J(y, u, w) &+ \int_Q \varphi \Big(\frac{\partial y}{\partial t}  + Ay +  f(x, t, y, w) - u - w  \Big)dxdt \nonumber \\
&+ \int_\Omega \varphi(0) [y(0) - y_0] dx + \int_Q eg(x,t, y, u, w)dxdt.
\end{align}

\begin{lemma}
    \label{6.01}
    Suppose that the assumptions of Theorem \ref{DinhLy} are satisfied. Then there exist  constants $\alpha > 0$ and $r_2>0$ such that 
    \begin{align}
    \label{Lemma4.2.1}
        D^2_{(y, u)}\mathcal{L}(\bar y, \bar u, \varphi, e, \bar w)\Big[(y - \bar y, u - \bar u), (y - \bar y, u - \bar u)  \Big] \ge \alpha \|u - \bar u\|^2_{L^2(Q)}
    \end{align}
    for all $y \in S_F^{u, \bar w}$ and $u \in B_U(\bar u, r_2)$.
\end{lemma}
\begin{proof} 
    With $y \in S_F^{u, \bar w}$, it's easy to see that if $u = \bar u$ then by the uniqueness of solution of the equation (\ref{P2})-(\ref{P3}), we have $y =\bar y$, hence, in this case (\ref{Lemma4.2.1}) is satisfied for all $\alpha > 0$. 

    We now define 
    \begin{align}
        \mathcal O := \bigg\{\Big(\frac{ y - \bar y}{ t}, \frac{ u - \bar u}{ t}  \Big): \ \ y \in S_F^{u, \bar w}, \ u \in B_U(\bar u, r_2), \      u \ne \bar u  \  {\rm and} \  t := \| u - \bar u\|_{L^2(Q)}   \bigg\}
    \end{align}
and   
\begin{align}
    \alpha := \mathop {\rm inf}\limits_{(\zeta, v) \in \mathcal{O}}D^2_{(y, u)}\mathcal{L}(\bar y, \bar u, \varphi, e, \bar w)[(\zeta, v), (\zeta, v)].
\end{align}
To prove (\ref{Lemma4.2.1}), it's sufficient to show that $\alpha > 0$. We do this as follows. By definition of infimum, there exists a sequence $(\zeta_n, v_n) \in \mathcal{O}$ with $\zeta_n= \frac{y_n-\bar y}{t_n}$ and $v_n= \frac{u_n-\bar u}{t_n}$, $y_n \in S_F^{u_n, \bar w}$, $u_n \in B_U(\bar u, r_2)$, $u_n \ne \bar u$,  $t_n=\|u_n- \bar u\|_{L^2(Q)}$, $n = 1, 2, ...$,  such that
\begin{align}
\label{StrictSOSCond.02}
    \alpha =  \mathop{\rm lim}\limits_{n \to \infty}D^2_{(y, u)}\mathcal{L}(\bar y, \bar u, \varphi, e, \bar w)[(\zeta_n, v_n), (\zeta_n, v_n)].
\end{align}
Since $\|v_n \|_{L^2(Q)}=1$, $\{v_n\}$ is bounded in $L^2(Q)$. By the reflexivity of  $L^2(Q)$, we may assume that $v_n \rightharpoonup v$ in $L^2(Q)$.  We want to show that $\zeta_n$ converges weakly to some $\zeta$ in $W^{1,2}(0, T; D, H)$. Since $y_n\in S_F^{u_n, \bar w}$ and $\bar y\in S_F^{\bar u, \bar w}$, we have $F(y_n, u_n, \bar w) = F(\bar y, \bar u, \bar w) = 0$, that is, 
\begin{align*}
&\frac{\partial y_n}{\partial t}+ Ay_n +f(x, t, y_n, \bar w) = u_n + \bar w,\quad y_n(0)=y_0\\
&\frac{\partial \bar y}{\partial t} +A\bar y + f(x, t, \bar y, \bar w)=\bar u + \bar w,\quad \bar y(0)= y_0.
\end{align*} 
This implies that
\begin{align}
\label{Zn}
     \frac{\partial (y_n-\bar y)}{\partial t}+ A(y_n-\bar y) +f(x, t,  y_n, \bar w)-f(x, t, \bar y, \bar w) = u_n-\bar u, \quad  (y_n-\bar y)(0)=0. 
\end{align}
 By a Taylor's expansion of function $y \mapsto f(x, t, y, \bar w)$, there exist measurable functions $\theta_n$ such that 
 \begin{align}
 \label{Zn.1}
f(x, t, y_n, \bar w) - f(x, t, \bar y, \bar w) = f_y\Big(x, t, \bar y +\theta_n(y_n-\bar y), \bar w \Big) (y_n-\bar y),\quad  0\leq\theta_n(x, t)\leq 1. 
\end{align}
 From \eqref{Zn} and (\ref{Zn.1}), we get
\begin{align}
\label{Zn2}
     \frac{\partial \zeta_n}{\partial t}+ A\zeta_n  + f_y\Big(x, t, \bar y+\theta_n (y_n-\bar y), \bar w\Big)\zeta_n = v_n, \quad \zeta_n(0)=0.
\end{align}
Also, 
\begin{align*}
    |u_n(x, t)| \le |u_n(x, t) - \bar u(x, t)| + |\bar u(x, t)| \le \|u_n - \bar u\|_U + \| \bar u\|_U =  \| \bar u\|_U +r_2
\end{align*}
for a.a. $(x, t) \in Q$. By $(i)$ of Lemma \ref{Lemma-stateEq}, there exist $c_1 > 0$ such that 
\begin{align}
    \|y_n\|_{L^\infty(Q)} \le c_1 (\|u_n\|_U + \|\bar w\|_W + \|y_0\|_{L^\infty(\Omega)}) \le c_1 (\|\bar u\|_U + \|\bar w\|_W + r_2   + \|y_0\|_{L^\infty(\Omega)}).
\end{align}
Hence, 
\begin{align*}
    |\bar y(x, t) + \theta_n(x, t) (y_n(x, t)-\bar y(x, t))| &\le \|\bar y+\theta_n (y_n-\bar y)\|_{L^\infty(Q)} \le 2 \|\bar y\|_{L^\infty(Q)} + \|y_n\|_{L^\infty(Q)} \\
    &\le 2 \|\bar y\|_{L^\infty(Q)} + c_1 (\|\bar u\|_U + \|\bar w\|_W + r_2   + \|y_0\|_{L^\infty(\Omega)})
\end{align*}
for a.a. $(x, t)\in Q$.  By assumption $(H2)$, we have 
\begin{align*}
    &\Big|f_y\Big(x, t, \bar y(x, t) + \theta_n(x, t) (y_n(x, t)-\bar y(x, t)), \bar w(x, t)\Big)\Big| \\ 
    &\le  k_{fM}\Big(|\bar y(x, t) + \theta_n(x, t) (y_n(x, t)-\bar y(x, t))| +  |\bar w(x, t)| \Big) + |f_y(x, t, 0, 0)| \\
    &\le  k_{fM}\Big(2 \|\bar y\|_{L^\infty(Q)} + c_1 (\|\bar u\|_U + \|\bar w\|_W + r_2   + \|y_0\|_{L^\infty(\Omega)}) + \|\bar w\|_W  \Big) + \|f_y(\cdot, \cdot, 0, 0)\|_{L^\infty(Q)}
\end{align*}
for a.a. $(x, t)\in Q$. Hence $f_y\Big(x, t, \bar y+\theta_n (y_n-\bar y), \bar w\Big) \in L^\infty(Q)$. From this, applying regularity result from \cite[Theorem 5, p.360]{Evan} to the parabolic equation (\ref{Zn2}), there exists a constant $C>0$ such that 
\begin{align}
\label{cmm.000}
\|\zeta_n\|_{W^{1,2}(0,T; D, H)}\leq  C\|v_n\|_{L^2(Q)}=C.
\end{align}
This implies that  $\{\zeta_n\}$ is bounded in $W^{1,2}(0,T; D, H)$. Without loss of generality, we may assume that $\zeta_n\rightharpoonup \zeta$ in $W^{1,2}(0,T, D, H)$. By Aubin's lemma, the embedding $W^{1,2}(0, T; D, H)\hookrightarrow L^2(0, T; V)$ is compact. Therefore, we may assume that $\zeta_n\to \zeta$ in $L^2(0, T; V)$. Particularly, $\zeta_n\to \zeta$ in $L^2(Q)$. We now  use the procedure in the proof of \cite[Theorem 3, p. 356]{Evan}. By  passing to the limit,   we obtain from  \eqref{Zn2} that
\begin{align}\label{Z-Eq-C2}
     \frac{\partial \zeta}{\partial t}+ A\zeta +f'(\bar y)\zeta = v,\ \zeta(0)=0.
\end{align}
 Therefore, we get $(\zeta, v)\in\mathcal{C}_2[(\bar y, \bar u), \bar w]$.  We now consider two cases.

 \noindent $\bullet$ \textit{Case 1: $\zeta = 0$.} \  \ Then we have from (\ref{StrictSOSCond.02}), (\ref{StrictSOSCond.0}), (\ref{StrictSOSCond.01}) and $\|v_n \|_{L^2(Q)}=1$ that 
 \begin{align*}
     \alpha &=  \mathop{\rm lim}\limits_{n \to \infty}D^2_{(y, u)}\mathcal{L}(\bar y, \bar u, \varphi, e, \bar w)[(\zeta_n, v_n), (\zeta_n, v_n)] \\
     &=  \mathop{\rm lim}\limits_{n \to \infty} \bigg\{\int_Q\Big[L_{yy}[x, t]\zeta_n^2 + 2L_{yu}[x, t]\zeta_nv_n  + e(g_{yy}[x, t]\zeta_n^2 + 2g_{yu}[x, t]\zeta_nv_n) + \varphi f_{yy}[x, t] \zeta_n^2\Big]dxdt  \bigg\} \\
     &+ \mathop{\rm lim}\limits_{n \to \infty} \int_Q\Big(L_{uu}[x, t] + e(x, t)g_{uu}[x, t]\Big)v_n^2dxdt  \\
     &\ge 0 + \varrho = \varrho > 0.
 \end{align*}

 \noindent $\bullet$ \textit{Case 2: $\zeta \ne 0$.} \  \ Then we have from the condition (\ref{StrictSOSCond.0}) that 
  \begin{align*}
     \alpha &=  \mathop{\rm lim}\limits_{n \to \infty}D^2_{(y, u)}\mathcal{L}(\bar y, \bar u, \varphi, e, \bar w)[(\zeta_n, v_n), (\zeta_n, v_n)] \\
     &=\int_Q(L_{yy}[x, t]\zeta^2 + 2L_{yu}[x, t]\zeta v  + L_{uu}[x, t]v^2)dxdt  \\
    &+ \int_Q e(x, t)\Big( g_{yy}[x, t]\zeta^2 + 2g_{yu}[x, t]\zeta v  + g_{uu}[x, t]v^2 \Big) dxdt  + \int_Q \varphi f_{yy}[x, t] \zeta^2 dxdt > 0.
 \end{align*}
 This completes the proof of Lemma.
\end{proof}

\begin{lemma}
    \label{6.02}   
    Under the assumptions of Lemma \ref{6.01}, there exist positive numbers $r'_2, \beta', \beta'' > 0$ such that the following holds: if $\varphi' \in B_{ L^\infty(Q)}(\varphi, \beta')$, $e' \in B_{L^\infty(Q)}(e, \beta'')$,  $y \in S_F^{u, \bar w}$ with $u \in B_U(\bar u, r'_2)$ then one has 
    \begin{align}
    \label{602.0}
        D^2_{(y, u)}\mathcal{L}(y, u,  \varphi',  e', \bar w)\Big[(y - \bar y, u - \bar u), (y - \bar y, u - \bar u)  \Big] \ge  \frac{\alpha}{2} \|u - \bar u\|^2_{L^2(Q)},
    \end{align}
    where $\alpha$ is given in  Lemma \ref{6.01}.
\end{lemma}
\begin{proof}  
    For $r_2 > 0$, let $h := y  - \bar y$, $v := u - \bar u$ with $u \in B_U(\bar u, r_2)$ and $y \in S_F^{u, \bar w}$.   According to Lemma \ref{6.01}, there exists a positive constant $\alpha > 0$ such that 
    \begin{align}
    \label{602.1}
        D^2_{(y, u)}\mathcal{L}(\bar y, \bar u, \varphi, e, \bar w)[(h, v), (h, v)] \ge \alpha \|v\|^2_{L^2(Q)}.
    \end{align}
    We set 
    \begin{align*}
        \|\cdot\|_{\varphi} := \|\cdot\|_{ L^\infty(Q)} \quad {\rm and} \quad  \|\cdot\|_{e} := \|\cdot\|_{L^\infty(Q)}.
    \end{align*}
    With $(\varphi', e') \in W_2^{1,1}(0, T; D, H) \cap L^\infty(Q) \times L^\infty(Q)$, we have 
    \begin{align}
    \label{602.2}
        &\Big |D^2_{(y, u)}\mathcal{L}(y, u, \varphi', e', \bar w)[(h, v), (h, v)]  -  D^2_{(y, u)}\mathcal{L}(\bar y, \bar u, \varphi, e, \bar w)[(h, v), (h, v)]  \Big| \nonumber \\
        &\le \Big| D^2_{(y, u)}J(y, u, \bar w)(h, v)^2 -  D^2_{(y, u)}J(\bar y, \bar u, \bar w)(h, v)^2 \Big| \nonumber \\
        &+ \Big| \int_Q[\varphi' f_{yy}(x, t, y, \bar w)h^2  - \varphi f_{yy}(x, t, \bar y, \bar w)h^2 ] dxdt\Big| \nonumber \\ 
        &+ \Big| \int_Q[e'D^2_{(y, u)}g(y, u, \bar w)(h, v)^2 -  eD^2_{(y, u)}g(\bar y, \bar u, \bar w)(h, v)^2] dxdt\Big|  =: P_1 + P_2 + P_3. 
    \end{align}
    Since $u \in B_U(\bar u, r_2)$, 
    \begin{align}
        |u(x, t)| \le |u(x, t) - \bar u(x, t)| + |\bar u(x, t)| \le \|u - \bar u\|_U + \|\bar u\|_U = \|\bar u\|_U + r_2
    \end{align}
    for a.a. $(x, t) \in Q$. Since $y \in S_F^{u, \bar w}$, $F(y, u, \bar w) = 0$. By  $(i)$ of Lemma \ref{Lemma-stateEq}, there exists a constant $c_1 > 0$ such that 
    \begin{align}
        \|y\|_{L^\infty(Q)} \le c_1(\| u\|_U + \|\bar w\|_W + |y_0\|_{L^\infty(\Omega)}) \le c_1(\|\bar u\|_U + r_2 + \|\bar w\|_W + |y_0\|_{L^\infty(\Omega)})
    \end{align}
    for a.a. $(x, t) \in Q$. Since $F(\bar y, \bar u, \bar w) = F(y, u, \bar w) = 0$, by using similar arguments as in the procedure in the proof of (\ref{cmm.000}), we obtain
    \begin{align}
        \label{602.3}
        \|h\|_{L^2(Q)} \le \|h\|_{W_2^{1,1}(0, T; D, H)} \le c_2\|v\|_{L^2(Q)}
    \end{align}
    for some constant $c_2 > 0$.      By assumption $(H3)$, we have 
    \begin{align}
        \label{602.4}
        P_1  &\le \Big|\int_Q[L_{yy}(x, t, y, u, \bar w) - L_{yy}(x, t, \bar y, \bar u, \bar w)]h^2 dxdt\Big| \nonumber \\ 
        &+ 2\Big|\int_Q[L_{yu}(x, t, y, u, \bar w) - L_{yu}(x, t, \bar y, \bar u, \bar w)]hv dxdt\Big| \nonumber \\
        &+ \Big|\int_Q[L_{uu}(x, t, y, u, \bar w) - L_{uu}(x, t, \bar y, \bar u, \bar w)]v^2 dxdt\Big| \nonumber \\
        &\le (c_2 + 1)^2k_{L, M}(\|y - \bar y\|_Y + \|u - \bar u\|_U)\|v\|^2_{L^2(Q)}
    \end{align}
    Moreover, from Lemma \ref{Lemma-stateEq.01}, there exists a constant $c_3 > 0$ such that 
\begin{align}
    \label{602.5}
    \|y - \bar y\|_Y  \le c_3 \|u - \bar u\|_U.
\end{align}
From (\ref{602.4}) and (\ref{602.5}), we get 
\begin{align}
    \label{602.6}
    P_1  \le (c_3 + 1)(c_2 + 1)^2k_{L, M}\|u - \bar u\|_U\|v\|^2_{L^2(Q)}.
\end{align}
For $P_2$, we do this as follows
\begin{align}
    \label{602.7}
    P_2 &\le \int_Q|\varphi' - \varphi| | f_{yy}(x, t,  y, \bar w) |h^2 dxdt  + \int_Q|\varphi| |f_{yy}(x, t, y, \bar w)  - f_{yy}(x, t, \bar y, \bar w) |h^2 dxdt  \nonumber \\
    &\le c_2^2(k_{f, M}\|\varphi' - \varphi\|_{\varphi} + \|\varphi\|_{\varphi}c_3 k_{f, M}\|u - \bar u\|_U)\|v\|^2_{L^2(Q)}.
\end{align}
We evaluate  $P_3$
\begin{align}
        \label{602.8}
        P_3  &\le \Big|\int_Q[e'g_{yy}(x, t, y, u, \bar w) - eg_{yy}(x, t, \bar y, \bar u, \bar w)]h^2 dxdt\Big| \nonumber \\ 
        &+ 2\Big|\int_Q[e'g_{yu}(x, t, y, u, \bar w) - eg_{yu}(x, t, \bar y, \bar u, \bar w)]hv dxdt\Big| \nonumber \\
        &+ \Big|\int_Q[e'g_{uu}(x, t, y, u, \bar w) - e g_{uu}(x, t, \bar y, \bar u, \bar w)]v^2 dxdt\Big| =: S_1 + S_2 + S_3.
    \end{align}
And 
\begin{align}
    \label{602.9}
    S_1 &\le \int_Q|e' - e| | g_{yy}(x, t,  y, u,  \bar w) |h^2 dxdt  + \int_Q|e| |g_{yy}(x, t, y, u,  \bar w)  - g_{yy}(x, t, \bar y, \bar u,  \bar w) |h^2 dxdt  \nonumber \\
    &\le c_2^2(k_{g, M}\|e' - e\|_e + \|e\|_e(c_3+ 1) k_{g, M}\|u - \bar u\|_U)\|v\|^2_{L^2(Q)}.
\end{align}
By similar arguments
\begin{align}
    \label{602.10}
    S_2 &\le \int_Q|e' - e| | g_{yu}(x, t,  y, u,  \bar w) |hv dxdt  + \int_Q|e| |g_{yu}(x, t, y, u,  \bar w)  - g_{yu}(x, t, \bar y, \bar u,  \bar w) |hv dxdt  \nonumber \\
    &\le 2c_2(k_{g, M}\|e' - e\|_e + \|e\|_e(c_3+ 1) k_{g, M}\|u - \bar u\|_U)\|v\|^2_{L^2(Q)},
\end{align}
and
\begin{align}
    \label{602.11}
    S_3 &\le \int_Q|e' - e| | g_{uu}(x, t,  y, u,  \bar w) |v^2 dxdt  + \int_Q|e| |g_{uu}(x, t, y, u,  \bar w)  - g_{uu}(x, t, \bar y, \bar u,  \bar w) |v^2 dxdt  \nonumber \\
    &\le (k_{g, M}\|e' - e\|_e + \|e\|_e(c_3+ 1) k_{g, M}\|u - \bar u\|_U)\|v\|^2_{L^2(Q)}.
\end{align}
From (\ref{602.8}), (\ref{602.9}), (\ref{602.10}) and (\ref{602.11}), we obtain
\begin{align}
    \label{602.12}
    P_3 \le (c_2+1)^2(k_{g, M}\|e' - e\|_e + \|e\|_e(c_3+ 1) k_{g, M}\|u - \bar u\|_U)\|v\|^2_{L^2(Q)}.
\end{align}
From (\ref{602.2}), (\ref{602.6}), (\ref{602.7}) and (\ref{602.12}), we obtain
   \begin{align}
    \label{602.13}
        &\Big |D^2_{(y, u)}\mathcal{L}(y, u, \varphi', e', \bar w)[(h, v), (h, v)]  -  D^2_{(y, u)}\mathcal{L}(\bar y, \bar u, \varphi, e, \bar w)[(h, v), (h, v)]  \Big| \nonumber \\
        &\le (c_3 + 1)(c_2 + 1)^2k_{L, M}\|u - \bar u\|_U\|v\|^2_{L^2(Q)} \nonumber \\
        &+ c_2^2(k_{f, M}\|\varphi' - \varphi\|_{\varphi} + \|\varphi\|_{\varphi}c_3 k_{f, M}\|u - \bar u\|_U)\|v\|^2_{L^2(Q)} \nonumber \\ 
        &+ (c_2+1)^2(k_{g, M}\|e' - e\|_e + \|e\|_e(c_3+ 1) k_{g, M}\|u - \bar u\|_U)\|v\|^2_{L^2(Q)}.
    \end{align}
We now choose positive numbers $r'_2, \beta', \beta'' > 0$ such that  if $\varphi' \in B_{ L^\infty(Q)}(\varphi, \beta')$, $e' \in B_{L^\infty(Q)}(e, \beta'')$ and $u \in B_U(\bar u, r'_2)$ then
\begin{align}
    (c_3 + 1)(c_2 + 1)^2k_{L, M}\|u &- \bar u\|_U + c_2^2(k_{f, M}\|\varphi' - \varphi\|_{\varphi} + \|\varphi\|_{\varphi}c_3 k_{f, M}\|u - \bar u\|_U) \nonumber \\
    &+ (c_2+1)^2(k_{g, M}\|e' - e\|_e + \|e\|_e(c_3+ 1) k_{g, M}\|u - \bar u\|_U) \le \frac{\alpha}{2}.
\end{align}
Then, we have 
\begin{align}
    \label{602.14}
        &\Big |D^2_{(y, u)}\mathcal{L}(y, u, \varphi', e', \bar w)[(h, v), (h, v)]  -  D^2_{(y, u)}\mathcal{L}(\bar y, \bar u, \varphi, e, \bar w)[(h, v), (h, v)]  \Big| \le \frac{\alpha}{2}\|v\|^2_{L^2(Q)}.
    \end{align}
    Combining (\ref{602.1}) with (\ref{602.14}), we get (\ref{602.0}). This completes the proof of Lemma. 
    \end{proof}

\medskip

By a similar way as the proof of Lemma \ref{6.02}, it is easy to show that the conclusion of Lemma \ref{6.02} is still true if in the right-hand side of (\ref{602.0}), $y$ is replaced by $y'$ and $u$ is replaced by $u'$, where $(y', u') \in B_Y(\bar y, \beta''') \times B_U(\bar u, r'_2)$ with $\beta'''$ and $ r'_2$ small enough. Namely, we have:
    
\begin{corollary}
    \label{6.02.1}
    Under the assumptions of Lemma \ref{6.02},    there exist positive numbers $r'_2, \beta', \beta'', \beta''' > 0$ such that the following holds: if $\varphi' \in B_{ L^\infty(Q)}(\varphi, \beta')$, $e' \in B_{L^\infty(Q)}(e, \beta'')$,  $y \in S_F^{u, \bar w}$ with $u \in B_U(\bar u, r'_2)$ then one has 
    \begin{align}
    \label{602.00}
        D^2_{(y, u)}\mathcal{L}(y', u',  \varphi',  e', \bar w)\Big[(y - \bar y, u - \bar u), (y - \bar y, u - \bar u)  \Big] \ge  \frac{\alpha}{2} \|u - \bar u\|^2_{L^2(Q)},
    \end{align}
    for all $(y', u') \in B_Y(\bar y, \beta''') \times B_U(\bar u, r'_2)$.
\end{corollary}

\begin{proposition}
    \label{6.03}
    Under  assumptions of Theorem \ref{DinhLy}, $\bar z = (\bar y, \bar u)$  is a locally strongly optimal solution of the problem $(P(\bar w))$ in the sense of (\ref{strong.solution}). 
\end{proposition}
\begin{proof}
    Take any $(y, u) \in \Phi(\bar w)$. We then notice that
    \begin{align*}
        F(y, u, \bar w) = F(\bar y, \bar u, \bar w) = 0; \quad e(x, t)g(x, t, \bar y, \bar u, \bar w) = 0, \  e(x, t)g(x, t,  y,  u, \bar w) \le 0  \  {\rm a.a.} \ (x, t) \in Q,
    \end{align*}
    where $(\varphi, e) \in \Lambda_\infty[(\bar y, \bar u), \bar w]$. Hence 
    \begin{align}
        \label{602.15}
        \mathcal{L}(y, u, \varphi, e, \bar w) - \mathcal{L}(\bar y, \bar u, \varphi, e, \bar w) &= J(y, u, \bar w) - J(\bar y, \bar u, \bar w) + \int_Qe(x, t)g(x, t,  y,  u, \bar w) dxdt \nonumber \\ 
        &\le  J(y, u, \bar w) - J(\bar y, \bar u, \bar w).
    \end{align}
    Moreover, by a Taylor expansion of function $z \mapsto \mathcal{L}(z, \varphi, e, \bar w)$ around $\bar z = (\bar y, \bar u)$, there exists a measurable function $z' = \bar z + \xi(z- \bar z)$, $0 \le \xi(x, t) \le 1$, such that 
    \begin{align}
        \label{602.16}
        \mathcal{L}(y, u, \varphi, e, \bar w) &- \mathcal{L}(\bar y, \bar u, \varphi, e, \bar w) \nonumber \\ 
        &=  D_{(y, u)}\mathcal{L}(\bar z, \varphi, e, \bar w)(z - \bar z) + \frac{1}{2}D^2_{(y, u)}\mathcal{L}(z', \varphi, e, \bar w)(z - \bar z)^2 \nonumber \\
        &= \frac{1}{2}D^2_{(y, u)}\mathcal{L}(z', \varphi, e, \bar w)(z - \bar z)^2.
    \end{align}
    By  Corollary  \ref{6.02.1}, there exist numbers $\epsilon>0$, $\kappa > 0$ and a  constant $\alpha_1 > 0$ such that 
    \begin{align}
        \label{602.17}
        D^2_{(y, u)}\mathcal{L}(z', \varphi, e, \bar w)(z - \bar z)^2 \ge \alpha_1 \|u - \bar u\|^2_{L^2(Q)},
    \end{align}
    for all $z =  (y, u)\in \Phi(\bar w) \cap[B_Y(\bar y, \epsilon)\times B_U(\bar u, \epsilon)]$. Combining (\ref{602.15}), (\ref{602.16}) with (\ref{602.17}), we obtain 
    \begin{align}
        J(y, u, \bar w) - J(\bar y, \bar u, \bar w) \ge \frac{\alpha_1}{2}\|u - \bar u\|^2_{L^2(Q)},
    \end{align}
    for all $ (y, u)\in \Phi(\bar w) \cap[B_Y(\bar y, \epsilon)\times B_U(\bar u, \epsilon)]$. Thus $\bar z$  is a locally strongly optimal solution of the problem $(P(\bar w))$. The lemma is proved. 
\end{proof}

\section{Proof of the main result}

Before proving Theorem \ref{DinhLy}, we need  the following lemmas. 

\begin{lemma}
\label{cmm1}
    Let the assumptions of Theorem \ref{DinhLy} be satisfied. Then
    \begin{align}
    \label{cmm2}
    \mathcal L (\bar y, \bar u, \varphi, e, \bar w) \ge \mathcal L (\bar y, \bar u, \varphi_w, e_w, \bar w),
\end{align}
where $(\varphi_w, e_w) \in  \Lambda_\infty [(\widehat y_w, \widehat u_w), w]$.
\end{lemma}
\begin{proof}
From the formula of the Lagrange function in (\ref{cmm0}),  we have 
\begin{align}
    \mathcal L (\bar y, \bar u, \varphi, e, \bar w) - \mathcal L (\bar y, \bar u, \varphi_w, e_w, \bar w) = - \int_Q e_w(x, t)g(x, t, \bar y, \bar u, \bar w) dxdt. 
\end{align}
Combining this with the facts that $e_w(x, t) \ge 0$ and $g(x, t, \bar y, \bar u, \bar w) \le 0$ a.a. $(x, t) \in Q$, we get (\ref{cmm2}). The Lemma is proved. 
\end{proof}

\begin{lemma}\label{cmmm} There exist positive constants $K_1, K_2, K_3 > 0$ such that:
    \begin{align}
        & \|\widetilde y - \bar y \|_{W_2^{1, 1}(0, T; D, H)} \le K_1\|\widehat u_w - \bar u\|_{L^2(Q)}, \label{cmmm.1} \\
        & \|\widetilde y - \widehat y_w\|_{W_2^{1, 1}(0, T; D, H)} \le K_2\|w - \bar w\|_{L^\infty(Q)}, \label{cmmm.2} \\
        &\|\widehat y_w - \bar y\|_{W_2^{1, 1}(0, T; D, H)} \le K_3\Big( \|\widehat u_w - \bar u\|_{L^2(Q)} + \|w - \bar w\|_{L^\infty(Q)}  \Big), \label{cmmm.3}
    \end{align}
    where $\widetilde y \in S_F^{\widehat u_w, \bar w}$.
\end{lemma}
\begin{proof}  We have that $\widetilde y, \bar y$ are solutions to the equation (\ref{P2})-(\ref{P3}) with $(u, w) = (\widehat u_w, \bar w)$, $(u, w) = (\bar u, \bar w)$, respectively. By using similar arguments as in the proof of (\ref{cmm.000}), we obtain (\ref{cmmm.1}). Notice that $F(\widetilde y, \widehat u_w, \bar w) = 0$ and $F(\widehat y_w, \widehat u_w,  w) = 0$. Using similar arguments as in the procedure in the proof of (\ref{3.9}), we obtain (\ref{cmmm.2}). Finally,  (\ref{cmmm.3}) is followed from (\ref{cmmm.1}) and (\ref{cmmm.2}). Therefore, the lemma is proved. 
\end{proof}

\medskip
 
\noindent {\bf Proof of Theorem \ref{DinhLy}}  
  
By Proposition \ref{6.03}, $\bar z =  (\bar y, \bar u)$  is a locally strongly optimal solution of the problem $(P(\bar w))$ in the sense of (\ref{strong.solution}). We next  show that the  estimates (\ref{KQC}) and (\ref{KQC.1})  are  satisfied for some $r_*, s_* > 0$. We do this as follows.

  Let $\widetilde y \in S_F^{\widehat u_w, \bar w}$. By Taylor's expansions of the Lagrange function, there exist measurable functions $\xi_k$, $0 \le \xi_k(x, t) \le 1$, $k = 1, 2$,  such that 
  \begin{align}
      &\mathcal L (\widetilde y, \widehat u_w, \varphi, e, \bar w) \nonumber  \\
        &=\mathcal L (\bar y, \bar u, \varphi, e, \bar w) + D_{(y, u)}\mathcal L (\bar y, \bar u, \varphi, e, \bar w)(\widetilde y - \bar y, \widehat u_w - \bar u) \nonumber \\ 
        &+  
        \frac{1}{2} D^2_{(y, u)}\mathcal L \Big((\bar y, \bar u) + \xi_1[(\widetilde y, \widehat u_w) - (\bar y, \bar u)], \varphi, e, \bar w\Big)[(\widetilde y - \bar y, \widehat u_w - \bar u), (\widetilde y - \bar y, \widehat u_w - \bar u)] \nonumber \\
        &= \mathcal L (\bar y, \bar u, \varphi, e, \bar w) \nonumber  \\ 
        &+  \frac{1}{2} D^2_{(y, u)}\mathcal L \Big((\bar y, \bar u) + \xi_1[(\widetilde y, \widehat u_w) - (\bar y, \bar u)], \varphi, e, \bar w\Big)[(\widetilde y - \bar y, \widehat u_w - \bar u), (\widetilde y - \bar y, \widehat u_w - \bar u)],    \label{7.1} \\
        {\rm and} \nonumber \\
         &\mathcal L (\bar y, \bar u, \varphi_w, e_w, \bar w) \nonumber \\
        &=\mathcal L (\widetilde y, \widehat u_w, \varphi_w, e_w, \bar w) - D_{(y, u)}\mathcal L (\widetilde y, \widehat u_w, \varphi_w, e_w, \bar w)(\widetilde y - \bar y, \widehat u_w - \bar u) \nonumber \\ 
        &+  
        \frac{1}{2} D^2_{(y, u)}\mathcal L \Big((\bar y, \bar u) + \xi_2[(\widetilde y, \widehat u_w) - (\bar y, \bar u)], \varphi_w, e_w, \bar w\Big)[(\widetilde y - \bar y, \widehat u_w - \bar u), (\widetilde y - \bar y, \widehat u_w - \bar u)].  \label{7.2}
  \end{align}
  Let $r'_2$ be positive number chosen according to Corollary \ref{6.02.1} such that 
  \begin{align}
      & \frac{1}{2} D^2_{(y, u)}\mathcal L \Big((\bar y, \bar u) + \xi_1[(\widetilde y, \widehat u_w) - (\bar y, \bar u)], \varphi, e, \bar w\Big)(\widetilde y - \bar y, \widehat u_w - \bar u)^2 \ge \alpha_1 \|\widehat u_w - \bar u\|^2_{L^2(Q)}  \label{7.3} \\
      & {\rm and} \nonumber \\
      & \frac{1}{2} D^2_{(y, u)}\mathcal L \Big((\bar y, \bar u) + \xi_2[(\widetilde y, \widehat u_w) - (\bar y, \bar u)], \varphi_w, e_w, \bar w\Big)(\widetilde y - \bar y, \widehat u_w - \bar u)^2 \ge  \alpha_2 \|\widehat u_w - \bar u\|^2_{L^2(Q)},  \label{7.4}
  \end{align}
  for some positive constants $\alpha_1, \alpha_2 > 0$.
  We put
  \begin{align}
  \label{7.5}
      r_* := {\rm min}(r_1, r'_2) \quad {\rm and} \quad s_* := s_1
  \end{align}
  where $r_1, s_1$ from Proposition \ref{4.03}.     
  
    In what follows: $(\widehat u_w, w) \in B_U(\bar u, r_*) \times B_W(\bar w, s_*)$ with $(r_*, s_*)$ defined by (\ref{7.5}).

    From (\ref{7.1}), (\ref{7.3}) and (\ref{cmm2}), we deduce that 
    \begin{align}
    \label{7.6}
    \mathcal L (\widetilde y, \widehat u_w, \varphi, e, \bar w) \ge \mathcal L (\bar y, \bar u, \varphi_w, e_w, \bar w).
\end{align}
Combining (\ref{7.6}), (\ref{7.2}) with (\ref{7.4}), we obtain 
    \begin{align}
        \label{7.7}
        \mathcal L (\widetilde y, \widehat u_w, \varphi, e, \bar w)  &\ge  \mathcal L (\widetilde y, \widehat u_w, \varphi_w, e_w, \bar w) \nonumber \\
        &- D_{(y, u)}\mathcal L (\widetilde y, \widehat u_w, \varphi_w, e_w, \bar w)(\widetilde y - \bar y, \widehat u_w - \bar u)   +  \alpha_2 \|\widehat u_w - \bar u\|^2_{L^2(Q)}.
    \end{align}
    It follows that 
    \begin{align}
    \label{cmm10}
        \alpha_2 \|\widehat u_w - \bar u\|^2_{L^2(Q)} &\le \Big[\mathcal L (\widetilde y, \widehat u_w, \varphi, e, \bar w) - \mathcal L (\widetilde y, \widehat u_w, \varphi_w, e_w, \bar w) \Big]  \nonumber \\
        &+ D_{(y, u)}\mathcal L (\widetilde y, \widehat u_w, \varphi_w, e_w, \bar w)(\widetilde y - \bar y, \widehat u_w - \bar u)   
        =: A_1 + A_2.        
    \end{align}
    Since  $F(\widetilde y, \widehat u_w, \bar w) = 0$, we have 
    \begin{align}
    \label{cmm11}
        A_1  = \mathcal L (\widetilde y, \widehat u_w, \varphi, e, \bar w) - \mathcal L (\widetilde y, \widehat u_w, \varphi_w, e_w, \bar w)  
        = \int_Q (e - e_w)g(x, t, \widetilde y, \widehat u_w, \bar w) dxdt.
    \end{align}
    Also, 
    \begin{align*}
        (e - e_w)g(x, t, \widetilde y, \widehat u_w, \bar w) &- (e - e_w)[g(x, t, \widetilde y, \widehat u_w, \bar w) - g(x, t, \widehat y_w, \widehat u_w, w)] \\
        &= e(x, t).g(x, t, \widehat y_w, \widehat u_w, w) \le 0, 
    \end{align*}
    which implies that
    \begin{align}
        \label{cmm12}
        (e - e_w)g(x, t, \widetilde y, \widehat u_w, \bar w) \le (e - e_w)[g(x, t, \widetilde y, \widehat u_w, \bar w) - g(x, t, \widehat y_w, \widehat u_w, w)]. 
    \end{align}
    From this and (\ref{cmm11}), we obtain 
    \begin{align}
    \label{cmm13}
        A_1  &\le \int_Q (e - e_w)[g(x, t, \widetilde y, \widehat u_w, \bar w) - g(x, t, \widehat y_w, \widehat u_w, w)] dxdt \nonumber \\
        &\le \Big|\int_Q (e - e_w)[g(x, t, \widetilde y, \widehat u_w, \bar w) - g(x, t, \widehat y_w, \widehat u_w, w)] dxdt \Big| \nonumber \\
        &\le \int_Q |e - e_w||g(x, t, \widetilde y, \widehat u_w, \bar w) - g(x, t, \widehat y_w, \widehat u_w, w)| dxdt.
    \end{align}
    Since $\widehat u_w \in B_U(\bar u, r_*)$ and  $w \in B_W(\bar w, s_*)$,
    \begin{align*}
        &|\widehat u_w(x, t)| \le \|\widehat u_w\|_{L^\infty(Q)} = \|(\widehat u_w - \bar u) + \bar u\|_{L^\infty(Q)} \le  \|\bar u\|_{L^\infty(Q)} +  \|\widehat u_w - \bar u\|_{L^\infty(Q)}  \le \|\bar u\|_{L^\infty(Q)} + r_*, \\
        &|w(x, t)| \le \|w\|_{L^\infty(Q)} = \|(w - \bar w) + \bar w\|_{L^\infty(Q)} \le \| \bar w\|_{L^\infty(Q)} + \|w - \bar w\|_{L^\infty(Q)}  \le \|\bar w\|_{L^\infty(Q)} + s_*
    \end{align*}
    for a.a. $(x, t) \in Q$.   Since  $F(\widetilde y, \widehat u_w, \bar w) = F(\widehat y_w, \widehat u_w,  w) = 0$, by $(i)$ of Lemma \ref{Lemma-stateEq}, there exist constants $c_{01}, c_{02} > 0$ such that 
    \begin{align*}
        &|\widetilde y(x, t)| \le \|\widetilde y\|_{L^\infty(Q)} \le c_{01}(\|\widehat u_w\|_{L^\infty(Q)} + \|\bar w\|_{L^\infty(Q)} + \|y_0\|_{L^\infty(\Omega)}) \\ 
        &\quad \quad \quad  \quad \le c_{01}(\|\bar u\|_{L^\infty(Q)} + r_* + \|\bar w\|_{L^\infty(Q)} + \|y_0\|_{L^\infty(\Omega)}), \\
        &|\widehat y(x, t)| \le \|\widehat y\|_{L^\infty(Q)} \le c_{02}(\|\widehat u_w\|_{L^\infty(Q)} + \| w\|_{L^\infty(Q)} + \|y_0\|_{L^\infty(\Omega)}) \\ 
        &\quad \quad \quad  \quad \le c_{02}(\|\bar u\|_{L^\infty(Q)} + r_* + \|\bar w\|_{L^\infty(Q)} + s_* + \|y_0\|_{L^\infty(\Omega)}), 
    \end{align*}
    for a.a. $(x, t) \in Q$.    By assumption  $(H3)$, we have 
    \begin{align}
    \label{cmm14}
        \Big |g(x, t, \widetilde y, \widehat u_w, \bar w) - g(x, t, \widehat y_w, \widehat u_w, w) \Big| \le k_{g, M} \Big(|\widetilde y(x, t) - \widehat y_w(x, t)| + |w(x, t) - \bar w(x, t)| \Big)
    \end{align}
    for a.a. $(x, t) \in Q$.   Combining (\ref{cmm13}) with (\ref{cmm14}), we get 
     \begin{align}
    \label{cmm15}
        A_1 &\le k_{g, M} \int_Q |e - e_w||\widetilde y(x, t) - \widehat y_w(x, t)| dxdt +  k_{g, M}  \int_Q |e - e_w| |w(x, t) - \bar w(x, t)|dxdt  \nonumber \\
        &\le k_{g, M} \|e - e_w\|_{L^2(Q)}  \Big(\|\widetilde y - \widehat y_w\|_{L^2(Q)}  + \|w - \bar w\|_{L^2(Q)} \Big).
    \end{align}
    From Lemma \ref{5.0}, there exists a constant $a_1 > 0$ such that 
    \begin{align}
    \label{cmm16}
        \|e - e_w\|_{L^2(Q)} \le a_1 \Big( \|\widehat u_w - \bar u\|_{L^2(Q)} + \|w - \bar w\|_{L^\infty(Q)} \Big).
    \end{align}
    From $(ii)$ of  Lemma \ref{cmmm}, there exists a constant $a_2 > 0$ such that 
    \begin{align}
    \label{cmm17}
       \|\widetilde y - \widehat y_w\|_{L^2(Q)}  \le a_2 \|w - \bar w\|_{L^\infty(Q)}.
    \end{align}
    It follows from (\ref{cmm15}), (\ref{cmm16}) and (\ref{cmm17}), there exists a positive constant $a_3 := a_1(a_2 + 1)k_{g, M} > 0$ such that 
    \begin{align}
        \label{cmm18}
        A_1 \le  a_3 \|w - \bar w\|_{L^\infty(Q)} \Big( \|\widehat u_w - \bar u\|_{L^2(Q)} + \|w - \bar w\|_{L^\infty(Q)} \Big).
    \end{align}

    To estimate $A_2$, we do as follows.  Since $(\varphi_w, e_w) \in \Lambda_\infty[(\widehat y_w, \widehat u_w), w]$,  $D_{(y, u)}\mathcal{L}(\widehat y_w, \widehat u_w, \varphi_w, e_w, w) = 0$. Hence, with $h := \widetilde y - \bar y$ and $v := \widehat u_w - \bar u$,  we have 
    
    \begin{align}\label{cmm19}
A_2  &\le   \Big|D_{(y, u)}\mathcal L (\widetilde y, \widehat u_w, \varphi_w, e_w, \bar w)(h, v) 
 -   D_{(y, u)}\mathcal{L}(\widehat y_w, \widehat u_w, \varphi_w, e_w, w)(h, v) \Big| \nonumber \\
 &\le   \Big|D_{(y, u)}\mathcal L (\widetilde y, \widehat u_w, \varphi_w, e_w, \bar w)(h, v) 
 -   D_{(y, u)}\mathcal L (\widehat y_w, \widehat u_w, \varphi_w, e_w, \bar w)(h, v) \Big| \nonumber \\
 &+ \Big|D_{(y, u)}\mathcal L (\widehat y_w, \widehat u_w, \varphi_w, e_w, \bar w)(h, v) 
 - D_{(y, u)}\mathcal{L}(\widehat y_w, \widehat u_w, \varphi_w, e_w, w)(h, v) \Big|  =: T_1 + T_2. 
    \end{align}
We have 
\begin{align*}
      &D_{(y, u)}\mathcal L (\widetilde y, \widehat u_w, \varphi_w, e_w, \bar w)(h, v) =D_{(y, u)} J (\widetilde y, \widehat u_w, \bar w)(h, v) \\
    &+  \int_Q \varphi_w \Big(\frac{\partial h}{\partial t}  + Ah +  f_y(x, t, \widetilde y, \bar w)h - v  \Big)dxdt  +  \int_Q e_wD_{(y, u)}g(\cdot,\cdot, \widetilde y, \widehat u_w, \bar w)(h, v)dxdt, \\
     &D_{(y, u)}\mathcal L (\widehat y_w, \widehat u_w, \varphi_w, e_w, \bar w)(h, v) =D_{(y, u)} J (\widehat y_w, \widehat u_w, \bar w)(h, v) \\
    &+  \int_Q \varphi_w \Big(\frac{\partial h}{\partial t}  + Ah +  f_y(x, t, \widehat y_w, \bar w)h - v  \Big)dxdt  +  \int_Q e_wD_{(y, u)}g(\cdot,\cdot, \widehat y_w, \widehat u_w, \bar w)(h, v)dxdt.
\end{align*}
Hence
\begin{align}
\label{cmm20.1}
    T_1  &\le  \Big|D_{(y, u)} J (\widetilde y, \widehat u_w, \bar w)(h, v) - D_{(y, u)} J (\widehat y_w, \widehat u_w, \bar w)(h, v) \Big| \nonumber \\
    &+  \int_Q |\varphi_w| |f_y(x, t, \widetilde y, \bar w) - f_y(x, t, \widehat y_w, \bar w)|.|h| dxdt \nonumber \\
    &+  \int_Q |e_w| |D_{(y, u)}g(\cdot,\cdot, \widetilde y, \widehat u_w, \bar w)(h, v) - D_{(y, u)}g(\cdot,\cdot, \widehat y_w, \widehat u_w, \bar w)(h, v)| dxdt.
\end{align}
According to Lemma \ref{5.0}, there exists a constant $c_{03} > 0$ such that  
\begin{align*}
     \|\varphi_w - \varphi\|_{L^\infty(Q)} + \|e_w - e\|_{L^\infty(Q)} \le c_{03}(\|\widehat u_w - \bar u\|_{L^\infty(Q)} +  \|w - \bar w\|_{L^\infty(Q)}) \le c_{03}(r_* + s_*). 
\end{align*}
Hence 
\begin{align*}
    &\|\varphi_w\|_{L^\infty(Q)} \le \|\varphi\|_{L^\infty(Q)} + \|\varphi_w - \varphi\|_{L^\infty(Q)} \le  \|\varphi\|_{L^\infty(Q)} +  c_{03}(r_* + s_*), \\
    &\|e_w \|_{L^\infty(Q)} \le  \| e\|_{L^\infty(Q)} + \|e_w - e\|_{L^\infty(Q)} \le \| e\|_{L^\infty(Q)} + c_{03}(r_* + s_*).
\end{align*}
Put 
\begin{align}
    M_0 := {\rm max} \bigg\{\|\varphi\|_{L^\infty(Q)}, \| e\|_{L^\infty(Q)}  \bigg\} + c_{03}(r_* + s_*).
\end{align}
Then, from (\ref{cmm20.1}), we have 
\begin{align}
\label{cmm20}
    T_1  &\le  \Big|D_{(y, u)} J (\widetilde y, \widehat u_w, \bar w)(h, v) - D_{(y, u)} J (\widehat y_w, \widehat u_w, \bar w)(h, v) \Big| \nonumber \\
    &+ M_0 \int_Q |f_y(x, t, \widetilde y, \bar w) - f_y(x, t, \widehat y_w, \bar w)|.|h| dxdt \nonumber \\
    &+ M_0 \int_Q |D_{(y, u)}g(\cdot,\cdot, \widetilde y, \widehat u_w, \bar w)(h, v) - D_{(y, u)}g(\cdot,\cdot, \widehat y_w, \widehat u_w, \bar w)(h, v)| dxdt \nonumber \\
    &=: B_1 + M_0B_2 + M_0B_3.
\end{align}
By Lemma \ref{cmmm} [estimate (\ref{cmmm.1}) and estimate (\ref{cmmm.2})], we have 
\begin{align*}
    \|h\|_{L^2(Q)} \le K_1\|v\|_{L^2(Q)} = K_1\|\widehat u_w - \bar u\|_{L^2(Q)} \quad {\rm and} \quad \|\widetilde y - \widehat y_w\|_{L^2(Q)} \le K_2\|w - \bar w\|_{L^\infty(Q)}. 
\end{align*}
Combining these facts with the assumption $(H3)$, we have 
\begin{align}
    \label{cmm21}
    B_1  &\le  \int_Q \Big|L_y(x, t, \widetilde y, \widehat u_w, \bar w) - L_y (x, t, \widehat y_w, \widehat u_w, \bar w) \Big|. |h|dxdt \nonumber \\
    &+     \int_Q \Big|L_u(x, t, \widetilde y, \widehat u_w, \bar w) - L_u (x, t, \widehat y_w, \widehat u_w, \bar w) \Big|. |v|dxdt \nonumber\\
    &\le k_{L, M} \int_Q | \widetilde y -  \widehat y_w |. |h|dxdt  + k_{L, M} \int_Q | \widetilde y -  \widehat y_w |. |v|dxdt \nonumber \\
    &\le k_{L, M} \|\widetilde y - \widehat y_w\|_{L^2(Q)} \Big(\|h\|_{L^2(Q)} +  \|v\|_{L^2(Q)} \Big) \nonumber \\
    &\le k_{L, M} K_2(1 + K_1)\|w - \bar w\|_{L^\infty(Q)}\|\widehat u_w - \bar u\|_{L^2(Q)} =: a_4 \|w - \bar w\|_{L^\infty(Q)}\|\widehat u_w - \bar u\|_{L^2(Q)}
\end{align}
and 
\begin{align}\label{cmm22}
    B_2 &\le k_{f, M} \int_Q | \widetilde y -  \widehat y_w |. |h|dxdt \le k_{f, M} \|\widetilde y - \widehat y_w\|_{L^2(Q)} \|h\|_{L^2(Q)} \nonumber \\ 
    &\le k_{f, M}K_1K_2\|w - \bar w\|_{L^\infty(Q)}\|\widehat u_w - \bar u\|_{L^2(Q)} 
    =: a_5 \|w - \bar w\|_{L^\infty(Q)}\|\widehat u_w - \bar u\|_{L^2(Q)},
\end{align} and 
\begin{align}\label{cmm23}
    B_3  &\le  \int_Q \Big|g_y(x, t, \widetilde y, \widehat u_w, \bar w) - g_y (x, t, \widehat y_w, \widehat u_w, \bar w) \Big|. |h|dxdt \nonumber \\
    &+     \int_Q \Big|g_u(x, t, \widetilde y, \widehat u_w, \bar w) - g_u (x, t, \widehat y_w, \widehat u_w, \bar w) \Big|. |v|dxdt \nonumber\\
    &\le k_{g, M}  \int_Q | \widetilde y -  \widehat y_w |. |h|dxdt  + k_{g, M}   \int_Q | \widetilde y -  \widehat y_w |. |v|dxdt \nonumber \\
    &\le  k_{g, M}  \|\widetilde y - \widehat y_w\|_{L^2(Q)} \Big(\|h\|_{L^2(Q)} +  \|v\|_{L^2(Q)} \Big) \nonumber \\
    &\le k_{g, M}  K_2(1 + K_1)\|w - \bar w\|_{L^\infty(Q)}\|\widehat u_w - \bar u\|_{L^2(Q)} =: a_6 \|w - \bar w\|_{L^\infty(Q)}\|\widehat u_w - \bar u\|_{L^2(Q)}.
\end{align}
From (\ref{cmm20}),  (\ref{cmm21}), (\ref{cmm22}) and (\ref{cmm23}), we obtain
\begin{align}
    \label{cmm24}
    T_1 &\le (a_4 + M_0a_5 + M_0a_6)\|w - \bar w\|_{L^\infty(Q)}\|\widehat u_w - \bar u\|_{L^2(Q)} \nonumber \\
    &=: a_7\|w - \bar w\|_{L^\infty(Q)}\|\widehat u_w - \bar u\|_{L^2(Q)}.
\end{align}
For estimating on $T_2$, by similar arguments as above, we also obtain 
\begin{align}
    \label{cmm25}
    T_2  \le  a_8\|w - \bar w\|_{L^\infty(Q)}\|\widehat u_w - \bar u\|_{L^2(Q)}
\end{align}
for some constants $a_8 > 0$. Combining (\ref{cmm19}), (\ref{cmm24}) and (\ref{cmm25}), we get
\begin{align}
    \label{cmm26}
    A_2 \le (a_7 + a_8)\|w - \bar w\|_{L^\infty(Q)}\|\widehat u_w - \bar u\|_{L^2(Q)} =:  a_9\|w - \bar w\|_{L^\infty(Q)}\|\widehat u_w - \bar u\|_{L^2(Q)}.
\end{align}
From (\ref{cmm10}), (\ref{cmm18}) and (\ref{cmm26})
\begin{align}
    \label{cmm27}
       \alpha_2 \|\widehat u_w - \bar u\|^2_{L^2(Q)}  &\le a_3 \|w - \bar w\|_{L^\infty(Q)} \Big( \|\widehat u_w - \bar u\|_{L^2(Q)} + \|w - \bar w\|_{L^\infty(Q)} \Big) \nonumber \\
       &+ a_9\|w - \bar w\|_{L^\infty(Q)}\|\widehat u_w - \bar u\|_{L^2(Q)}.
\end{align}
Putting
\begin{align*}
    X := \|\widehat u_w - \bar u\|_{L^2(Q)},  \quad  a :=   \alpha_2 > 0,  \quad b := (a_3 + a_9)\|w - \bar w\|_{L^\infty(Q)},  \quad c := - a_3 \|w - \bar w\|^2_{L^\infty(Q)}. 
\end{align*}
Then (\ref{cmm27}) becomes 
\begin{align}
    \label{cmm28.0}
    aX^2 - bX + c \le 0. 
\end{align}
It follows that 
\begin{align}
    \label{cmm28}
    X \le X_2 := \frac{b + \sqrt \Delta }{2a}, \quad {\rm where} \quad \Delta := b^2 - 4ac. 
\end{align}
In fact that  $\Delta = b^2 - 4ac = [(a_3 + a_9)^2 + 4a_3  \alpha_2]\|w - \bar w\|^2_{L^\infty(Q)}.$  Hence
\begin{align}
\label{cmm29}
    \|\widehat u_w - \bar u\|_{L^2(Q)} =  X \le X_2 = \widehat K_{\rm Lips}.\|w - \bar w\|_{L^\infty(Q)}, 
\end{align}
where
\begin{align*}
    \widehat K_{\rm Lips} := \frac{a_3 + a_9 + \sqrt{(a_3 + a_9)^2 + 4a_3  \alpha_2}}{2 \alpha_2}.
\end{align*}
From (\ref{cmmm.3}) and (\ref{cmm29}), we obtain (\ref{KQC}). From (\ref{5.00.1}), (\ref{5.10}) and (\ref{cmm29}), we obtain (\ref{KQC.1}). Therefore,  Theorem \ref{DinhLy} is proved. $\hfill\square$

\medskip

\noindent {\bf Acknowledgments}   
The research  was funded by the  International Centre for Research and Postgraduate Training in Mathematics, Institute of Mathematics, VAST, under grant number ICRTM02-2022.01.

\end{document}